\documentclass[10pt,leqno]{amsart}
\usepackage{amsfonts,amssymb}
\usepackage{tikz, subfigure, xcolor} 
\usepgfmodule{decorations}
\usetikzlibrary{decorations, patterns, snakes}

\usepackage{amsmath}
\usepackage{amsthm}
\usepackage{enumerate}
\usepackage{hyperref, doi}

\usepackage[noadjust]{cite}
\usepackage{comment}

\usepackage{color}

\newtheorem{theorem}{Theorem}[section]

\newtheorem{lemma}[theorem]{Lemma}
\newtheorem{proposition}[theorem]{Proposition}

\theoremstyle{definition}

\newtheorem{remark}[theorem]{Remark}

\theoremstyle{plain}

\newcommand{\IR}{\mathbb{R}}

\newcommand{\IN}{\mathbb{N}}

\newcommand{\IE}{\mathbb{E}}
\newcommand{\IP}{\mathbb{P}}

\newcommand{\cO}{\mathcal{O}}

\newcommand{\X}{\mathrm{X}}

\newcommand{\norm}[1]{\left\Vert#1\right\Vert}

\newcommand{\divergence}{\operatorname{div}}

\numberwithin{equation}{section} 

\title[The three limits of the hydrostatic approximation]{The three limits of the hydrostatic approximation}

\author[Furukawa]{Ken Furukawa}
\address{Faculty of Science, Academic Assembly, University of Toyama,
	Gofuku 3190, Toyama, 930-0887, Toyama, Japan}
\email{furukawa@sci.u-toyama.ac.jp}

\author[Giga]{Yoshikazu Giga} 
\address{Graduate School of Mathematical Sciences, University of Tokyo, Komaba 3-8-1, Meguro-ku, Tokyo, 153-8914, Japan}
\email{labgiga@ms.u-tokyo.ac.jp}
\email{tkashiwa@ms.u-tokyo.ac.jp}

\author[Hieber]{Matthias Hieber} 

\address{Department of Mathematics,
	TU Darmstadt, Schlossgartenstr. 7, 64289 Darmstadt, Germany}
\email{hieber@mathematik.tu-darmstadt.de}
\email{marc.p.wrona@gmail.com}

\author[Hussein]{Amru Hussein} 
\address{Department of Mathematics,
	RPTU Kaiserslautern-Landau,  Paul-Ehrlich-Stra\ss e 31, 67663 Kaiserslautern, Germany}
\email{hussein@mathematik.uni-kl.de}

\author[Kashiwabara]{Takahito Kashiwabara}

\author[Wrona]{Marc Wrona}


\thanks{
	Yoshikazu Giga was partly supported by the Japan Society for the Promotion of Sciences (JSPS) through grants Kakenhi: Nos.20K20342, 19H00639, 18H05323, and by Arithmer Inc., Daikin Industries, Ltd. and Ebara Corporation through collaborative grants. Takahito Kashiwabara was supported by JSPS Grant-in-Aid for Early-Career Scientists No. 20K14357.
	Matthias Hieber gratefully acknowledges the support by the Deutsche Forschungsgemeinschaft (DFG) through  the Research Unit 5528. Amru Hussein has been generously supported by 
	DFG -- project number 508634462 and by MathApp -- Mathematics Applied to Real-World Problems -- part of the Research Initiative of the Federal State of Rhineland-Palatinate, Germany. 
}

\begin{document}

\begin{abstract}
	The primitive equations are derived from the $3D$-Navier-Stokes equations by the hydrostatic approximation. Formally, assuming an $\varepsilon$-thin domain and anisotropic viscosities with vertical viscosity $\nu_z=\mathcal{O}(\varepsilon^\gamma)$ where $\gamma=2$, one obtains the primitive equations with full viscosity as $\varepsilon\to 0$. Here, we take two more limit equations into consideration: For $\gamma<2$ the $2D$-Navier-Stokes equations are obtained. For $\gamma>2$ the primitive equations with only horizontal viscosity $-\Delta_H$ as $\varepsilon\to 0$. Thus, there are three possible limits of the hydrostatic approximation depending on the assumption on the vertical viscosity. 
	The latter convergence has been proven recently by Li,  Titi, and Yuan 
	using energy estimates. Here, we consider more generally $\nu_z=\varepsilon^2 \delta$ and show how maximal regularity methods and quadratic inequalities can be an efficient approach to the same end for $\varepsilon,\delta\to 0$.  The flexibility of our methods is also illustrated  by the convergence for $\delta\to \infty$ and $\varepsilon\to 0$ to the $2D$-Navier-Stokes equations.    
\end{abstract}

\maketitle
\addtocontents{toc}{\protect\setcounter{tocdepth}{1}}
\tableofcontents

\section{Introduction}\label{sec:intro}
The hydrostatic approximation of the Navier-Stokes equations is used in meteorological models, cf. e.g. \cite{Pedlosky1979, Vallis17, Washington, Majda}. Formally, it consists in replacing the equation for the vertical velocity of the fluid by the assumption that the pressure is determined by the surface pressure. It appears already in the pioneering work of Richardson originally dating back to 1922, see \cite{richardson_lynch_2007} for a recent edition, and it can be justified on physical grounds by a scale analysis. 
Mathematically, it can be justified by a rescaling of Navier-Stokes equations with anisotropic viscosities. That is, one considers on a vertically $\varepsilon$--thin domain
\begin{align*}
\Omega_{\varepsilon}:=G \times (-\varepsilon, \varepsilon), \quad \hbox{where} \quad G:=(-1,1)\times (-1,1), \quad \varepsilon>0,
\end{align*} 
the Navier-Stokes equations   
\begin{equation*}
\left \{\begin{array}{rll}
\partial _tu+u\cdot\nabla u-\Delta_H u - \nu_z(\varepsilon) \partial_z^2 u+\nabla p=& \ 0&\hbox{ in }(0,T)\times\Omega_\varepsilon,\\
\divergence u=& \ 0&\hbox{ in }(0,T)\times\Omega_\varepsilon,\\
u(0)=& u_0&\hbox{ in }\Omega_\varepsilon,\\
\end{array}\right .\tag{NS$^{\Omega_{\varepsilon}}$}
\end{equation*}
where the vertical viscosity coefficient $\nu_z(\varepsilon)$ is $\varepsilon$ dependent with $\nu_z(\varepsilon)\to 0$ as $\varepsilon\to 0$.  
Rescaling the velocity $u=(v,w)$ with horizontal components $v=(v_1,v_2)$ vertical component $w$, and the pressure $p$ by   
\begin{align*}
v_{\varepsilon,\delta}(x,y,z):= v(x,y,\varepsilon z), \quad
w_{\varepsilon,\delta}(x,y,z):= \tfrac{1}{\varepsilon} w(x,y,\varepsilon z), \quad \hbox{and} \quad
p_{\varepsilon,\delta}(x,y,z):= p(x,y,\varepsilon z),
\end{align*}
equation~\eqref{eq:NS_nu} transforms to the rescaled anisotropic Navier-Stokes equations on the $\varepsilon$-independent domain $\Omega:=\Omega_1$
\begin{equation}\label{eq:NS_eps_Intro}
\left \{\begin{array}{rll}
\partial _tv_{\varepsilon,\delta}+u_{\varepsilon,\delta}\cdot\nabla v_{\varepsilon,\delta}-\Delta_{H} v_{\varepsilon,\delta}
-\frac{\nu_z(\varepsilon)}{\varepsilon^2}\partial_z^2 v_{\varepsilon,\delta}
+\nabla _Hp_{\varepsilon,\delta}=& \ 0&\text{ in }(0,T)\times\Omega ,\\
\varepsilon^2\left(\partial _tw_{\varepsilon,\delta}+u_{\varepsilon,\delta}\cdot \nabla w_{\varepsilon,\delta}-\Delta_{H} w_{\varepsilon,\delta}-\frac{\nu_z(\varepsilon)}{\varepsilon^2}\partial_z^2w_{\varepsilon,\delta}\right) +\partial _zp_{\varepsilon,\delta}=&\ 0&\text{ in }(0,T)\times\Omega ,\\
\divergence u_{\varepsilon,\delta} =& \ 0&\text{ in }(0,T)\times\Omega ,\\
u_{\varepsilon,\delta} (0)=&\ (u_0)_{\varepsilon,\delta}&\text{ in }\Omega,
\end{array}\right .\tag{NS$_{\varepsilon,\delta}$}
\end{equation}
where the $w$-equation has been multiplied by $\varepsilon^2$. This recaling procedure goes back at least to \cite{Azerad, BessonLaydi1992}. 
In meteorological modeling anisotropic and partial viscosities appear naturally, cf. e.g. \cite{Chemin_2006}, since the viscosity is largely an eddy viscosity, cf. e.g. \cite{PTZ, Washington}.


The formal limit equation for $\varepsilon\to 0$ now depends crucially on the behaviour of the term
$\frac{\nu_z(\varepsilon)}{\varepsilon^2}$. More generally, we consider 
\begin{align}\label{eq:nudeltaeps}
\nu_z=\varepsilon^2\delta \quad \hbox{for}\quad \varepsilon,\delta>0,\quad \hbox{that is} \quad \frac{\nu_z(\varepsilon)}{\varepsilon^2}=\delta.
\end{align}
For $\varepsilon\to 0$ one hence has three cases: For $\delta>0$ constant setting for simplicity $\delta=1$, one has formally the primitive equations with full viscosity as $\frac{\nu_z(\varepsilon)}{\varepsilon^2}=1$, compare \cite{convergence, Azerad, MR3926040}. If $\delta \to 0$, one obtains the primitive equations with only horizontal viscosity as in \eqref{eq:NS_eps_Intro} also the term $\frac{\nu_z(\varepsilon)}{\varepsilon^2}\partial_z^2v_\varepsilon\to 0$ for $\delta\to 0$. 
The case $\delta\to \infty$ is not that straight forward. A first indication is the energy equality obtained by testing  \eqref{eq:NS_eps_Intro} with $(v_\varepsilon,  w_\varepsilon)$
\begin{align*}
\norm{(v_{\varepsilon}(t),\varepsilon w_\varepsilon)}_{L^2}^2 + \int_{0}^t \norm{\nabla_H( v_{\varepsilon}(s),  \varepsilon w_\varepsilon)}_{L^2}^2
+ \delta\norm{\partial_z (v_{\varepsilon}(s), \varepsilon w_\varepsilon)}_{L^2}^2 ds =	\norm{((v_0)_{\varepsilon}, (\varepsilon w_0)_{\varepsilon})}_{L^2}^2, \quad t>0.
\end{align*}
This implies that $\partial_z v_{\varepsilon}\to 0$ in $L^2(0,T;L^2(\Omega))$ as $\delta \to \infty$, that is, heuristically in the limit only the horizontal directions are relevant for the horizontal velocity. To understand the limit behaviour, one splits the anisotropic Navier-Stokes equations into 
barotropic and baroclinic modes, respectively, i.e., 
\begin{align*}
{u}_{\varepsilon,\delta}=\overline{u}_{\varepsilon,\delta}+\tilde{u}_{\varepsilon,\delta}, \quad \hbox{where}\quad
\overline{u}_{\varepsilon,\delta}:= \overline{v}_{\varepsilon,\delta}  \quad \hbox{and} \quad 
\tilde{u}_{\varepsilon,\delta}:= (\tilde{v}_{\varepsilon,\delta}, w_{\varepsilon,\delta}) 
\end{align*}
with
\begin{align*}
\overline{v}_{\varepsilon,\delta}:= \frac{1}{2}\int_{-1}^{+1} v_{\varepsilon,\delta}(\cdot,\cdot,\xi) d\xi
\quad \hbox{and} \quad 
\tilde{v}_{\varepsilon,\delta}:= {v}_{\varepsilon,\delta}-\overline{v}_{\varepsilon,\delta}.
\end{align*}
Then
one obtains for $\overline{v}_{\varepsilon,\delta}$ the $2D$-Navier-Stokes equations with forcing
\begin{equation*}
\left \{\begin{array}{rll}
\partial _t\overline{v}_{\varepsilon,\delta}+\overline{v}_{\varepsilon,\delta}\cdot\nabla \overline{v}_{\varepsilon,\delta}-\Delta_{H} \overline{v}_{\varepsilon,\delta}+\nabla _H\overline{p}_{\varepsilon,\delta}=& 
\ - \tfrac{1}{2}\int_{-1}^{+1}\tilde{u}_{\varepsilon,\delta}\cdot\nabla \tilde{v}_{\varepsilon,\delta}
&\text{ in }(0,T)\times G,\\
\divergence_H \overline{v}_{\varepsilon,\delta} =& \ 0&\text{ in }(0,T)\times G ,\\
\overline{v}_{\varepsilon,\delta} (0)=&\ (\overline{v}_0)_{\varepsilon,\delta}&\text{ in }G,
\end{array}\right .
\end{equation*}
where using that $\divergence \tilde{u}_{\varepsilon,\delta}=0$ and assuming $\tilde{v}_{\varepsilon,\delta}$ to be even with respect to the vertical direction one writes
\begin{align*}
- \tfrac{1}{2}\int_{-1}^{+1}\tilde{u}_{\varepsilon,\delta}\cdot\nabla \tilde{v}_{\varepsilon,\delta}
=- \tfrac{1}{2}\int_{-1}^{+1} \tilde{v}_{\varepsilon,\delta}\cdot\nabla_H \tilde{v}_{\varepsilon,\delta}
-w_{\varepsilon,\delta}\partial_z \tilde{v}_{\varepsilon,\delta}
=\tfrac{1}{2}\int_{-1}^{+1} -\tilde{v}_{\varepsilon,\delta}\cdot\nabla_H \tilde{v}_{\varepsilon,\delta}+\divergence_{H}\tilde{v}_{\varepsilon,\delta}\cdot\tilde{v}_{\varepsilon,\delta}.
\end{align*}
Here, by the Poincar\'e inequality $\partial_z \tilde{v}_{\varepsilon,\delta}\to 0$ implies formally $\tilde{v}_{\varepsilon,\delta}\to 0$ as $\delta \to \infty$.  Therefore -- assuming that $\nabla_H \tilde{v}_{\varepsilon,\delta}$ remains bounded -- the forcing term vanishes leaving on a formal level the $2D$-Navier-Stokes equations as limit equations as $\delta\to \infty$. 

This formal reasoning will be made rigorous in Section~\ref{sec:quadratic_inequality}. 
Here, we study these three limits as sketched in Figure~\ref{fig:epsdelta}, see also Remark~\ref{rem:convergence} below. The full set of equations discussed is given in detail in Section~\ref{sec:eq}.

Depending on the assumptions on the scaling of the vertical viscosity, there can also be other limit equations  as in the case of the great lake equation with only vertical viscosity in the limit as discussed by Bresch,  Lemoine and Simon in  \cite{Bresch1999}. 
There, isotropic $\varepsilon$-independent viscosities are assumed and a different re-scaling is applied. 
Another setting is where the 3D-Navier-Stokes equations on a thin domain $\Omega_\varepsilon$ are considered, however with constant $\varepsilon$-independent viscosity. On such domains comparisons to  2D-Navier-Stokes-type equations can be shown, and thereby -- under certain assumptions on the data -- the global well-posedness of the 3D-Navier-Stokes equations on such thin domains follows, see e.g. \cite{MR1634572, MR1268906, MR1319876, MR1179539, MR2333468, MR2319923} and the references therein. The main difference to our approach is that we consider anisotropic viscosities and therefore rescaled rather than actual 3D Navier-Stokes equations as limiting equations. In particular, we do not have to restrict the class of the data.    
Indeed as pointed out in \cite[1. Introduction]{Azerad}, the anisotropic viscosity hypothesis
is fundamental for the derivation of the primitive equations.
For an illustrative overview on the model hierarchy and the various approximations for geophysical models see \cite[Figure 1.1]{Luesink_PhD}.

Instead of the general assumption \eqref{eq:nudeltaeps} on $\nu_z$ in \eqref{eq:NS_eps_Intro},
one can consider more concretely 
\begin{align*}
\nu_z(\varepsilon)=\varepsilon^{\gamma}\quad \hbox{for}\quad \gamma>0, \quad \hbox{that is}\quad \delta_\varepsilon=\varepsilon^{\gamma-2}.
\end{align*}
This has been studied recently by Li, Titi and Yuan in \cite{LiTitiGuozhi_2022} for $\gamma>2$ which corresponds to the case $\delta \to 0$ here. A comparison between the results obtained in \cite{LiTitiGuozhi_2022} and here is given in Remark~\ref{rem:comp} below.    
One conclusion from the above considerations -- summarized in Figure~\ref{fig:epsdelta} -- is that if there is now some uncertainty as to how to model the parameter $\gamma$, then the case where the limit is the primitive equations with full viscosity is only a borderline case. The more stable limits under variation of $\gamma$ are in fact the primitive equations with only horizontal viscosity and the $2D$-Navier-Stokes equations. This is summarized in Figure~\ref{fig:gamma}, and
it  is one motivation to 
prove norm convergence results for these three cases in appropriate norms as elaborated in this note.

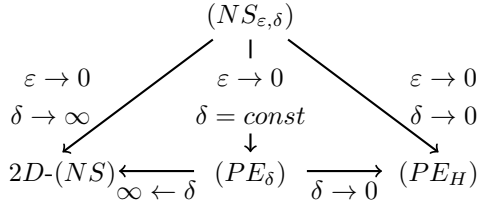
\begin{figure}[h]
	\begin{tikzpicture}[scale=0.5]
	\node[above] at (0,0) {$(NS_{\varepsilon,\delta})$};	
	\draw[thick, -] (0,0) -- (0,-0.5);
	\draw[thick, ->] (0,-2.5) -- (0,-3);	
	\node[below] at (0,-3) {$(PE_\delta)$};	
	\node at (0,-1) {$\varepsilon\to 0$};	
	\node at (0,-2) {$\delta=const$};
	\draw[thick, ->] (-1,0) -- (-5,-3);			
	\node[below] at (-5,-3) {$2D$-$(NS)$};	
	\node[left] at (-4,-1) {$\varepsilon\to 0$};	
	\node[left] at (-4,-2) {$\delta\to \infty$};
	\draw[thick, ->] (1,0) -- (5,-3);			
	\node[below] at (5,-3) {$(PE_H)$};	
	\node[right] at (4,-1) {$\varepsilon\to 0$};	
	\node[right] at (4,-2) {$\delta\to 0$};
	\draw[thick, ->] (-1.5,-3.5) -- (-3.5,-3.5);	
	\node[below] at (-2.5,-3.5) {$\infty \gets\delta$};					
	
	\draw[thick, ->] (1.5,-3.5) -- (3.5,-3.5);	
	\node[below] at (2.5,-3.5) {$\delta \to 0$};					
	\end{tikzpicture}
	\caption{Convergences schematically}\label{fig:epsdelta}
\end{figure}


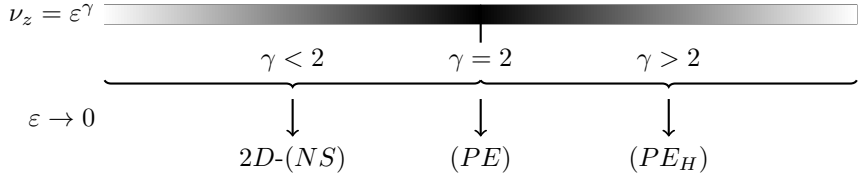
\begin{figure}[h]
	\begin{tikzpicture}[scale=0.5]
		\draw (0, 0) rectangle (10, 0.5); \
		\shade[left color=black, right color=white]
		(0,0) rectangle (10,0.5);
		\draw (0, 0) rectangle (-10, 0.5); \
		\shade[left color=white, right color=black]
		(0,0) rectangle (-10,0.5);
		\draw[thick] (0,0.5) -- (0,-0.5);
		\draw [
		thick,
		decoration={
			brace,
			mirror,
			raise=0.5cm
		},
		decorate
		] (0,-0.5) -- (10,-0.5);
		\node[above] at (5,-1.5) {$\gamma>2$};	
		\node[left] at (-10,0.25) {$\nu_z=\varepsilon^\gamma$};	
		\draw [
		thick,
		decoration={
			brace,
			mirror,
			raise=0.5cm
		},
		decorate
		] (-10,-0.5) -- (-0,-0.5);
		\node[above] at (-5,-1.5) {$\gamma<2$};	
		\node[above] at (0,-1.5) {$\gamma=2$};	
		\node[left] at (-10,-2.5) {$\varepsilon\to 0$};	
		\node[below] at (-5,-3) {$2D$-$(NS)$};	
		\draw[thick, ->] (-5,-2) -- (-5,-3);	
		\node[below] at (0,-3) {$(PE)$};	
		\draw[thick, ->] (0,-2) -- (0,-3);	
		\node[below] at (5,-3) {$(PE_H)$};	
		\draw[thick, ->] (5,-2) -- (5,-3);			
	\end{tikzpicture}
	\caption{Convergences schematically for $\nu_z=\varepsilon^\gamma$ as $\varepsilon\to 0$}\label{fig:gamma}
\end{figure}




\subsection{Literature on the  limit equations}
For the three limit equations in Figure~\ref{fig:epsdelta} global strong well-posedness is known. For the $2D$-Navier-Stokes equations the global well-posedness for initial data in $L^2$ is known since the early works by Leray, see \cite{Leray1934} and also e.g. \cite[Chapter V]{Sohr_book} and the references therein. 

The mathematical analysis of the primitive equations with full viscosity started much later, and it has been initiated in a series of papers by Lions, Temam and Wang \cite{Lionsetall1992, Lionsetall1992_b, Lionsetall1993}. There the primitive equations have been derived formally from the Navier-Stokes equations, and the existence of global weak solutions has been proven for initial data in $L^2$. A breakthrough result in the analysis of these equations has been the work by Cao and Titi, cf. \cite{CaoTiti2007}, who proved the global strong well-posedness for initial data in $H^1$. Since then there has been some refinements of this result enlarging the set of possible initial values, compare e.g. \cite{MR3592073, GigaGriesHusseinHieberKashiwabara2017NN, GGHHK_2020_a, Ju2017}, while the question of global well-posedness for initial data in $L^2$ remains open, see also \cite{boutros2023nonuniqueness} for recent developments in this direction.

The primitive equations with only horizontal viscosity has been studied by Cao, Li and Titi \cite{CaoLiTiti_PHE_2017, MR4117887}, see also \cite{HusseinSaalWrona} for an adaptation thereof,  proving that this equations are globally strongly well-posed. This can be contrasted with results on the ill-posedness  and blow-ups for the inviscid primitive equations, cf. e.g. \cite{MR3509003, Caoetall2015} and the references therein, and also \cite{Saal2018} for very weak viscosity terms which still guarantee at least local well-posedness.   
The primitive equations with only horizontal viscosity are part of a range of problems. When taking the temperature into account different types of anisotropic and partial viscosity and diffusivity can be imposed, see \cite{MR3237881, MR3264417, MR4117887}.
Here, we neglect the temperature and focus on the mathematically most challenging part which is the velocity equations. 

For the $3D$-Navier-Stokes equations the problem of the global well-posedness remains open, and there has been recent developments hinting towards blow ups, see \cite{Colombo2022, Merle_1, Merle_2}, and non-uniqueness, see  \cite{Buckmaster2017}. 
 As a side effect of our convergence results, it turns out that the  anisotropic $3D$-Navier-Stokes become globally well-posed when its solutions approaches one of the globally well-posed limit equations. 
%
%
%
%
%

%
\subsection{Literature on the hydrostatic approximation and our approach to it}
The scaling procedure outlined above 
had already been used by  Besson, Laydi and Touzani  in \cite{BessonLaydi1990}
for the stationary linear case and then by Besson and Laydi in \cite {BessonLaydi1992} for the stationary non-linear case.
The question of the convergence of the scaled Navier-Stokes equations to the primitive equations with full viscosity has been addressed first
 by Az\'{e}rad and Guill\'{e}n in \cite{Azerad} using compactness arguments and therefore without explicit convergence rate, where the linear time dependent problem has been studied before in \cite {Azerad_phD}. 
In the work by Li and Titi \cite{MR3926040} for the first time  strong convergence  was proven with convergence rate of order $\mathcal{O}(\varepsilon)$ based on energy estimates.
Later and taking a different approach by using maximal $L_t^q$-$L^p_x$-regularity methods and quadratic estimates, we proved in \cite{convergence} norm convergence of the same order for a large set of $p,q\in (1,\infty)$, see also \cite{convergence2} for the more difficult case of Dirichlet boundary conditions.  Non-periodic domains are included also in  \cite{Donatelli_2022}.   The cases of the scaled Boussinesq equations is considered in \cite{Pu_2023_a, Pu_2023_b}, the case of the compressible primitive equations is discussed in \cite{Sarka_2022} and the inviscid case in \cite{Sarka_2023}.

Here, we aim to adapt the methods developed in \cite{convergence} to prove, depending on the scaling of the vertical viscosity, convergence of the anisotropic Navier-Stokes equations to the primitive equations with only horizontal viscosity and to the $2D$-Navier Stokes equations, respectively. 
The problem of the convergence of the scaled Navier-Stokes equations to the primitive equations with only horizontal viscosity has been discussed recently by Li, Titi, and Yuan in \cite{LiTitiGuozhi_2022}.
The result obtained for $\nu_z=\varepsilon^\gamma$ for $\gamma>2$ there rely on energy methods, and they give 
convergence in the weak and strong sense with explicit convergence rates. The results are quite close to the results obtained here as discussed in Remark~\ref{rem:comp} below. 


The strength of the approach presented here is that it is easily adaptable which we show by including general $\nu_z=\varepsilon^2\cdot\delta$ instead of only $\delta_{\varepsilon}=\varepsilon^{\gamma-2}$ and by including the convergences $\varepsilon,\delta\to 0$ and $\varepsilon\to 0,\delta\to \infty$.
On a technical level the starting point is -- as e.g. in \cite{LiTitiGuozhi_2022, convergence, convergence2, MR3926040}  -- to consider the difference of the approximating and the limit equations. The main idea of 
our strategy is then to prove quadratic like inequalities for certain norms. These quadratic like inequalities are reminiscent to the approach by Fujita and Kato for the Navier-Stokes equations in \cite{FujitaKato}. There smallness of the data or the existence time implies contraction properties.
Here,  smallness enters the  inequalities in terms of the  parameters $\varepsilon, \delta$, and the quadratic inequalities imply uniform estimates with respect to these parameters, see Section~\ref{sec:quadratic_inequality} for a detailed discussion.

Originally, we developed this approach for the case of the hydrostatic approximation for the primitive equations with full viscosity in \cite{convergence}. To adapt the approach to the situation here, we have to make some adjustments reflecting the parabolic-hyperbolic character of the primitive equations with only horizontal viscosity. First, in order to obtain quadratic inequalities in maximal $L^2$-regularity norms and to compare this to \eqref{eq:NS_eps_Intro}, we add the missing term $\delta \partial^2_z v$ to both sides of the primitive equations with only horizontal viscosity given in Subsection~\ref{subsec:PEH} below. To continue with our strategy we then have to prove  uniform  estimates on  $\delta \partial^2_z v$ to estimate the right hand side. Second, to control the vanishing vertical viscosity on the left hand side, we linearise the difference equation as a diffusion-transport equation  with parabolic properties in horizontal directions and transport features in the vertical direction. Thereby we can prove that the vertical regularity of solutions is preserved and that it can be estimated in terms of the right hand side entering our quadratic like inequality.  
 
The convergence of the anisotropic Navier-Stokes equations for $\varepsilon\to 0$ and $\delta\to \infty$ can be proven following a similar overall strategy which is actually closer to \cite{convergence} and needs less adjustments. That all three limits of the hydrostatic approximation can be treated by this same general strategy  
illustrates that this is indeed a very  flexible  tool  to prove norm convergences. The hydrostatic approximation is just one of the mechanisms in the derivation of models in fluid mechanics,  see \cite[Figure 1.1]{Luesink_PhD} for an illustration, and
the methods presented here might be helpful to prove relevant convergences for other models.

\subsection{Outline} We start by discussing in the subsequent Section~\ref{sec:eq} the rescaled Navier-Stokes equations, its limit equations, and the equations satisfied by their differences. These difference equations are the main object studied here. In Section~\ref{sec:functionspaces} the framework of function spaces used here is introduced. 
In Section~\ref{sec:main} we present our main convergence results Theorems~\ref{thm:main1} and \ref{thm:main2} summarized in Remark~\ref{rem:convergence}. 
Their proof is subdivided into several steps: In Section~\ref{sec:lin}, we study linearisation for the difference equations, and in Section~\ref{sec:non_lin} non-linear estimates are derived. This is used in Section~\ref{sec:additional_reg} to show that assuming more regularity on the data, then the solution to the primitive equations with horizontal viscosity has additional regularity properties. In Section~\ref{sec:quadratic_inequality} we give a general scheme to conclude form quadratic inequalities uniform norm estimates. The results form the previous sections are then put together to apply this to the situation here and to prove Theorems~\ref{thm:main1} and \ref{thm:main2}.

\section{Rescaling procedure and the formal limit equations}\label{sec:eq}
In this section we introduce the precise setting for the 
limit and limiting equations treated, and 
for the relevant differences of the approximating and the limit equations.

\subsection{Rescaled $3D$ Navier-Stokes equations}
Consider the cylindrical $\varepsilon$-thin domain 
\begin{align}\label{eq:Omega}
\Omega_{\varepsilon}:=G \times (-\varepsilon, \varepsilon), \quad \hbox{where} \quad G:=(-1,1)\times (-1,1), \quad \varepsilon>0,
\end{align}
setting $$\Omega:=\Omega_1.$$ The horizontal coordinates are denoted by $(x,y)\in (-1,1)\times (-1,1)$,  the vertical coordinate by $z\in (-\varepsilon, \varepsilon)$,
and a finite time interval $(0,T)$ is given  with $T\in (0,\infty)$.
 Then, one considers the Navier-Stokes equations on $\Omega_\varepsilon$ with anisotropic viscosity given in terms of the horizontal viscosity constant $\nu_H>0$ and vertical viscosity constant $\nu_z>0$, and subject to periodic boundary conditions along with some parity in vertical direction, i.e., the velocity
  $u=(v,w)\colon \Omega_\varepsilon \times (0,T)\rightarrow \IR^3$ with $v=(v_1,v_2)$ and the pressure $p\colon \Omega_\varepsilon \times (0,T) \rightarrow \IR$ satisfy
\begin{equation}\label{eq:NS_nu}
\left \{\begin{array}{rll}
\partial _tu+u\cdot\nabla u-\nu_H\Delta_H u - \nu_z \partial_z^2 u+\nabla p=& \ 0&\hbox{ in }(0,T)\times\Omega_\varepsilon,\\
\divergence u=& \ 0&\hbox{ in }(0,T)\times\Omega_\varepsilon,\\
 u(0)=& u_0&\hbox{ in }\Omega_\varepsilon,\\
p, v, w \hbox{ periodic in }x,y,z,& \\
v,p \hbox{ even in }z,&\\
w \hbox{ odd in }z.&\\
\end{array}\right .\tag{NS$^{\Omega_{\varepsilon}}$}
\end{equation}
Here, we assume that $\nu_H=\cO(1)$ and that $\nu_z=\cO(\varepsilon^2 \delta)$, and for simplicity we set with $\varepsilon>0$ as  in \eqref{eq:Omega}
\begin{align*}
	\nu_H=1 \quad \hbox{and} \quad \nu_z = \varepsilon^2\cdot \delta \quad \hbox{for}\quad \delta>0.
\end{align*}
Here and in the following, we use the  notations
\begin{align*}
	\Delta_\delta = \Delta_H + \delta \partial_z^2, \quad \Delta_H=\partial_x^2+\partial_y^2,\quad \nabla_H=\begin{bmatrix} \partial_x \\ \partial_y\end{bmatrix}, \quad
	\divergence_H =\begin{bmatrix} \partial_x & \partial_y\end{bmatrix}.
\end{align*}
Then, one introduces the rescaled functions   $u_{\varepsilon,\delta}:=(v_{\varepsilon,\delta}, w_{\varepsilon,\delta})$ and $p_{\varepsilon,\delta}$ with
\begin{align*}
v_{\varepsilon,\delta}(x,y,z):= v(x,y,\varepsilon z), \quad
w_{\varepsilon,\delta}(x,y,z):= \tfrac{1}{\varepsilon} w(x,y,\varepsilon z), \quad \hbox{and} \quad
p_{\varepsilon,\delta}(x,y,z):= p(x,y,\varepsilon z),
\end{align*}
and thereby
equation~\eqref{eq:NS_nu} transforms to the rescaled anisotropic Navier-Stokes equations on $\Omega=\Omega_1$
\begin{equation}\label{eq:NS_eps}
\left \{\begin{array}{rll}
\partial _tv_{\varepsilon,\delta}+u_{\varepsilon,\delta}\cdot\nabla v_{\varepsilon,\delta}-\Delta_{\delta} v_{\varepsilon,\delta}+\nabla _Hp_{\varepsilon,\delta}=& \ 0&\text{ in }(0,T)\times\Omega ,\\
\partial _tw_{\varepsilon,\delta}+u_{\varepsilon,\delta}\cdot \nabla w_{\varepsilon,\delta}-\Delta_{\delta} w_{\varepsilon,\delta}+\tfrac{1}{\varepsilon ^2}\partial _zp_{\varepsilon,\delta}=&\ 0&\text{ in }(0,T)\times\Omega ,\\
\divergence u_{\varepsilon,\delta} =& \ 0&\text{ in }(0,T)\times\Omega ,\\
u_{\varepsilon,\delta} (0)=&\ (u_0)_{\varepsilon,\delta}&\text{ in }\Omega, \\
p_{\varepsilon,\delta}, v_{\varepsilon,\delta},w_{\varepsilon,\delta}, \text{ periodic in }x,y,z,& &\\
p_{\varepsilon,\delta}, v_{\varepsilon,\delta}\text{ even in } z,& & \\
w_{\varepsilon,\delta}\text{ odd in } z.& &
\end{array}\right .\tag{NS$_{\varepsilon,\delta}$}
\end{equation}
The even and odd parity conditions together with the periodicity imply the boundary conditions 
\begin{align*}
	\partial_zv_{\varepsilon,\delta}(\cdot, \cdot, \pm \varepsilon)=0, \quad \partial_zv_{\varepsilon,\delta}(\cdot, \cdot, 0)=0 \quad \hbox{and} \quad w_{\varepsilon,\delta}(\cdot, \cdot, \pm \varepsilon)=0, \quad  w_{\varepsilon,\delta}(\cdot, \cdot, 0)=0.
\end{align*}

\subsection{The anisotropic primitive equations}
Now, multiplying the equation for $w_{\varepsilon,\delta}$ in \eqref{eq:NS_eps} by $\varepsilon^2$, the anisotropic primitive equations are obtained by taking the formal limit of \eqref{eq:NS_eps} as $\varepsilon\to 0$ while fixing $\delta>0$. This means
we search for functions $u_{0,\delta}=(v_{0,\delta},w_{0,\delta})$ and $p_{0,\delta}$ satisfying 
\begin{equation}\label{eq_PE}
\left \{\begin{array}{rll}
\partial _tv_{0,\delta}+u_{0,\delta}\cdot\nabla v_{0,\delta}-\Delta_\delta v_{0,\delta}+\nabla _Hp_{0,\delta}=& \ 0&\text{ in }(0,T)\times\Omega ,\\
\partial _zp_{0,\delta}=&\ 0&\text{ in }(0,T)\times\Omega ,\\
\divergence u_{0,\delta}=& \ 0&\text{ in }(0,T)\times\Omega ,\\
v_{0,\delta}(0)=&\ (v_{0,\delta})_0&\text{ in }\Omega, \\
p_{0,\delta}, v_{0,\delta} ,w_{0,\delta} \text{ periodic in }x,y,z,& &
\\
v_{0,\delta}\text{ even in } z,& & \\
w_{0,\delta}\text{ odd in } z.& &
\end{array}\right .\tag{PE$_\delta$}
\end{equation}
The main difference in the structure of \eqref{eq_PE} as compared to \eqref{eq:NS_eps} is that
there is no evolution equations for the vertical velocity. Instead it can be recovered from the divergence free condition
and the parity conditions as
\begin{align}\label{eq:w}
w_{0,\delta}=w_{0,\delta}(v_{0,\delta})\quad \hbox{with} \quad	w_{0,\delta}(x,y,z)= -\int_{-1}^z \divergence_{H} 	v_{0,\delta}(x,y,\xi) d\xi,
\end{align}   
and therefrom the parity conditions on $w_{0,\delta}$ imply that the condition $\divergence u_{0,\delta}$ can be substituted by
\begin{align}\label{eq:divHv}
\divergence_{H} \int_{-1}^{1}v_{0,\delta}=0,
\end{align} 
compare e.g. \cite{Hieber2016} for a detailed discussion.
\subsection{The  primitive equations with only horizontal viscosity}\label{subsec:PEH}
Multiplying in \eqref{eq:NS_eps} the equation for $w_{\varepsilon,\delta}$ by $\varepsilon^2$, one can take the formal limit of \eqref{eq:NS_eps} as both $\varepsilon\to 0$ and $\delta\to 0$. Or, one takes the limit $\delta\to 0$ in \eqref{eq_PE}. This results in the primitive equations with only horizontal viscosity, i.e., functions
 $u_{0,0}=(v_{0,0},w_{0,0})$ and $p_{0,0}$ satisfying
\begin{equation}\label{eq_PEH}
\left \{\begin{array}{rll}
\partial _tv_{0,0}+u_{0,0}\cdot\nabla v_{0,0}-\Delta_H v_{0,0}+\nabla _Hp_{0,0}=& \ 0&\text{ in }(0,T)\times\Omega ,\\
\partial _zp_{0,0}=&\ 0&\text{ in }(0,T)\times\Omega ,\\
\divergence u_{0,0}=& \ 0&\text{ in }(0,T)\times\Omega ,\\
v_{0,0}(0)=&\ (v_{0,0})_0&\text{ in }\Omega, \\
p_{0,0}, v_{0,0} ,w_{0,0} \text{ periodic in }x,y,z,& &
\\
v_{0,0},\text{ even in } z,& & \\
w_{0,0}\text{ odd in } z.& &
\end{array}\right .\tag{PE$_H$}
\end{equation}
Here, the only formal difference to \eqref{eq_PE} is that $\Delta_\delta$ is substituted by $\Delta_H$. 
This has some severe implication for its analysis, because while \eqref{eq_PE} is a parabolic equation \eqref{eq_PEH} has parabolic diffusion features with respect to the horizontal variables while it has hyperbolic transport-like behaviour with respect to the vertical variable. 
The relations \eqref{eq:w} and \eqref{eq:divHv} carry over to the case $\delta=0$.

\subsection{Difference equations for $\varepsilon\to 0$ and $\delta\to 0$}
Let $u_{\varepsilon,\delta}=(v_{\varepsilon,\delta}, w_{\varepsilon,\delta})$ and $p_{\varepsilon,\delta}$ be the solution to ~\eqref{eq:NS_eps} and $v=v_{0,0}$ and $p=p_{0,0}$ be the solution to \eqref{eq_PEH} with $w=w(v)=w_{0,0}$, where we omit the subscripts to simplify the notation. Then multiplying the equations in \eqref{eq:NS_eps} and \eqref{eq_PEH} for the vertical components $w_{\varepsilon,\delta}$ and $w$ of the velocity by $\varepsilon$, respectively, one considers the differences
\begin{align}\label{eq:VW}
V_{\varepsilon,\delta} :=v_{\varepsilon,\delta} -v, \quad W_{\varepsilon,\delta}=\varepsilon(w_{\varepsilon,\delta} -w), \quad \hbox{and} \quad P_{\varepsilon,\delta} :=p_{\varepsilon,\delta} -p
\end{align}
setting $U_{\varepsilon,\delta}=(V_{\varepsilon,\delta},W_{\varepsilon,\delta})$, and one introduces the scaled gradient and divergence
\begin{align*}
\nabla_{\varepsilon} = \begin{bmatrix}
\partial_x \\ \partial_y \\ \tfrac{1}{\varepsilon}\partial_z
\end{bmatrix} \quad \hbox{and} \quad \divergence_{\varepsilon} = \begin{bmatrix}
\partial_x & \partial_y & \tfrac{1}{\varepsilon}\partial_z
\end{bmatrix}.
\end{align*} 
The functions $U_{\varepsilon,\delta}=(V_{\varepsilon,\delta},W_{\varepsilon,\delta})$ and $P_{\varepsilon,\delta}$ satisfy  the following anisotropic Navier-Stokes-type equations
\begin{equation}\label{eq_Diff}
\left \{\begin{array}{rll}
\partial _tV_{\varepsilon,\delta}- \Delta_{\delta}V_{\varepsilon,\delta}+ \nabla_H P_{\varepsilon,\delta}=&  \ F^H_{\varepsilon,\delta}(V_{\varepsilon,\delta}, W_{\varepsilon,\delta}) &\text{ in }(0,T)\times\Omega ,\\
\partial _tW_{\varepsilon,\delta}- \Delta_{\delta} W_{\varepsilon,\delta}+\frac{1}{\varepsilon}\partial_z P_{\varepsilon,\delta}=&  \ F^z_{\varepsilon,\delta}(V_{\varepsilon,\delta}, W_{\varepsilon,\delta}) &\text{ in }(0,T)\times\Omega ,\\
\divergence_\varepsilon U_{\varepsilon,\delta} =& \ 0&\text{ in }(0,T)\times\Omega ,\\
U_{\varepsilon,\delta}(0)=& \ 0&\text{ in }\Omega, \\
P_{\varepsilon,\delta}, V_{\varepsilon,\delta} ,W_{\varepsilon,\delta} \text{ periodic in }x,y,z,& &
\\
V_{\varepsilon,\delta}, P_{\varepsilon,\delta} \text{ even in }z,& &\\
W_{\varepsilon,\delta} \text{ odd in }z,& &
\end{array}\right .\tag{Diff$_{\varepsilon,\delta}$}
\end{equation}
where similarly as in \cite[Equation (3.1)]{convergence}
\begin{align}\label{eq:F}
\begin{split}
F^H_{\varepsilon,\delta}(V_{\varepsilon,\delta}, W_{\varepsilon,\delta}):=&- ( V_{\varepsilon,\delta}, \tfrac{1}{\varepsilon}W_{\varepsilon,\delta}) \cdot \nabla v 
-u \cdot \nabla V_{\varepsilon,\delta}
-( V_{\varepsilon,\delta}, \tfrac{1}{\varepsilon}W_{\varepsilon,\delta}) \cdot \nabla V_{\varepsilon,\delta}
+ \delta\partial_z^2 v,\\
F_{\varepsilon,\delta}^z(V_{\varepsilon,\delta}, W_{\varepsilon,\delta}):=&- (  V_{\varepsilon,\delta}, \tfrac{1}{\varepsilon} W_{\varepsilon,\delta}) \cdot \nabla \varepsilon w 
-u \cdot \nabla W_{\varepsilon,\delta}
-(V_{\varepsilon,\delta}, \tfrac{1}{\varepsilon} W_{\varepsilon,\delta}) \cdot \nabla W_{\varepsilon,\delta} \\
&- \varepsilon(\partial_t w + u \cdot \nabla w - \Delta_\delta w).
\end{split}
\end{align}
Notice that by the divergence free condition and the parity conditions on $W_{\varepsilon,\delta}$ one can rewrite
\begin{align}\label{eq:Weps} 
 \tfrac{1}{\varepsilon}\partial_z W_{\varepsilon,\delta} =  -\divergence_H V_{\varepsilon,\delta} \quad \hbox{and}\quad 
  \tfrac{1}{\varepsilon} W_{\varepsilon,\delta}=-\int_{-1}^z \divergence_H V_{\varepsilon,\delta}
\end{align}
provided that $U_{\varepsilon,\delta}$ is regular enough.
One can  rewrite  \eqref{eq:F} -- using $\divergence (V_{\varepsilon,\delta}, \tfrac{1}{\varepsilon}W_{\varepsilon,\delta})=0$ and $\divergence (v,w)=0$ -- to become
\begin{align}\label{eq:Fdivform}
\begin{split}
\begin{bmatrix}
F^H_{\varepsilon,\delta}(V_{\varepsilon,\delta}, W_{\varepsilon,\delta}) \\
F_{\varepsilon,\delta}^z(V_{\varepsilon,\delta}, W_{\varepsilon,\delta})
\end{bmatrix}
=& -\divergence \left((v, \varepsilon w)\otimes ( V_{\varepsilon,\delta}, \tfrac{1}{\varepsilon}W_{\varepsilon,\delta})\right)
-\divergence \left(( V_{\varepsilon,\delta}, W_{\varepsilon,\delta})\otimes (v,  w) \right)
\\
&-\divergence \left(( V_{\varepsilon,\delta}, W_{\varepsilon,\delta}) \otimes ( V_{\varepsilon,\delta}, \tfrac{1}{\varepsilon}W_{\varepsilon,\delta})\right)
+
\begin{bmatrix}
+\delta\partial_z^2 v
\\
- \varepsilon(\partial_t w + u \cdot \nabla w - \Delta_\delta w)
\end{bmatrix}.
\end{split}
\end{align}
The terms $\delta\partial_z^2 v$ and $- \varepsilon(\partial_t w + u \cdot \nabla w - \Delta_\delta w)$ on the right hand side of \eqref{eq_Diff} have been added on both sides of the equation, and they act as forcing terms which vanish as $\delta\to 0$ and $\varepsilon\to 0$. This will provide us with the ``smallness'' needed to show convergence via quadratic inequalities in Section~\ref{sec:quadratic_inequality}.

\subsection{Rescaled $3D$ Navier-Stokes equations with $2D$ Navier-Stokes equations as subsystem}
Consider now the limit $\delta\to \infty$ and $\varepsilon\to 0$.
Equation~\eqref{eq:NS_eps} can equivalently be rewritten to contain the $2D$ Navier-Stokes equations as a subsystem.
To this end, 
 split solutions to \eqref{eq:NS_eps} into the vertical average and the vertically average free part 
\begin{align*}
\overline{v}_{\varepsilon,\delta}&:= \frac{1}{2}\int_{-1}^{+1}v_{\varepsilon,\delta}(.\cdot,\cdot,\xi) d\xi \quad \hbox{and} \quad 
\tilde{v}_{\varepsilon,\delta}:= {v}_{\varepsilon,\delta}-\overline{v}_{\varepsilon,\delta}, \\
\overline{p}_{\varepsilon,\delta}&:= \frac{1}{2}\int_{-1}^{+1} p_{\varepsilon,\delta}(\cdot,\cdot,\xi) d\xi \quad \hbox{and} \quad 
\tilde{p}_{\varepsilon,\delta}:= {p}_{\varepsilon,\delta}-\overline{p}_{\varepsilon,\delta},
\end{align*}
respectively. This induces a splitting into barotropic and baroclinic modes, respectively, i.e., 
\begin{align*}
{u}_{\varepsilon,\delta}=\overline{v}_{\varepsilon,\delta}+\tilde{u}_{\varepsilon,\delta}
\quad \hbox{and}\quad
{p}_{\varepsilon,\delta}=\overline{p}_{\varepsilon,\delta}+\tilde{p}_{\varepsilon,\delta}, 
\quad \hbox{where}\quad 
\tilde{u}_{\varepsilon,\delta}:= (\tilde{v}_{\varepsilon,\delta}, w_{\varepsilon,\delta}). 
\end{align*}
Here, due to the parity condition $\int_{-1}^{+1} w_{\varepsilon,\delta}(\cdot,\cdot,\xi) d\xi=0$, the vertical velocity  $w_{\varepsilon,\delta}$ contributes to the baroclinic modes.
Then \eqref{eq:NS_eps} can be reformulated to become
\begin{equation}\label{eq:NS_eps2}
\left \{\begin{array}{rll}
\partial _t\overline{v}_{\varepsilon,\delta}+\overline{v}_{\varepsilon,\delta}\cdot\nabla \overline{v}_{\varepsilon,\delta}-\Delta_{H} \overline{v}_{\varepsilon,\delta}+\nabla _H\overline{p}_{\varepsilon,\delta}=& 
\ \overline{F}(\tilde{u}_{\varepsilon,\delta})
&\text{ in }(0,T)\times G,\\
\partial _t\tilde{v}_{\varepsilon,\delta}
-\Delta_{\delta} \tilde{v}_{\varepsilon,\delta}+\nabla _H\tilde{p}_{\varepsilon,\delta}=&
\ \tilde{F}_1(\tilde{u}_{\varepsilon,\delta}, \overline{v}_{\varepsilon,\delta}) 
&\text{ in }(0,T)\times\Omega ,\\
\partial _tw_{\varepsilon,\delta}-\Delta_{\delta} w_{\varepsilon,\delta}+\tfrac{1}{\varepsilon ^2}\partial _z\tilde{p}_{\varepsilon,\delta}=&
\ \tilde{F}_2(\overline{v}_{\varepsilon,\delta}, \tilde{u}_{\varepsilon,\delta})
&\text{ in }(0,T)\times\Omega ,\\
\divergence_H \overline{v}_{\varepsilon,\delta} =& \ 0&\text{ in }(0,T)\times G ,\\
\divergence \tilde{u}_{\varepsilon,\delta} =& \ 0&\text{ in }(0,T)\times\Omega, \\
\tilde{u}_{\varepsilon,\delta} (0)=&\ (\tilde{u}_0)_{\varepsilon,\delta}&\text{ in }\Omega,\\
\overline{v}_{\varepsilon,\delta} (0)=&\ (\overline{v}_0)_{\varepsilon,\delta}&\text{ in }G, \\
\int_{-1}^1\tilde{v}_{\varepsilon,\delta}=0, \quad \int_{-1}^1\tilde{p}_{\varepsilon,\delta}= & \ 0 &\text{ in }G, \\
\overline{p}_{\varepsilon,\delta}, \overline{v}_{\varepsilon,\delta}, \text{ periodic in }x,y,& &\\
\tilde{p}_{\varepsilon,\delta}, \tilde{v}_{\varepsilon,\delta}, w_{\varepsilon,\delta}\text{ periodic in }x,y,z,& &\\
\tilde{p}_{\varepsilon,\delta}, \tilde{v}_{\varepsilon,\delta}\text{ even in }z,&  & \\
w_{\varepsilon,\delta}\text{ odd in }z,& &
\end{array}\right .\tag{NS$_{\varepsilon,\delta}^{(2)}$}
\end{equation}
where
 \begin{align*}
 	\overline{F}(\tilde{u}_{\varepsilon,\delta})&:=
 - \tfrac{1}{2}\int_{-1}^{+1}\tilde{u}_{\varepsilon,\delta}\cdot\nabla \tilde{v}_{\varepsilon,\delta}, \\
 \tilde{F}_1(\overline{v}_{\varepsilon,\delta}, \tilde{u}_{\varepsilon,\delta}) &:=
  -\tilde{v}_{\varepsilon,\delta}\cdot\nabla_H \overline{v}_{\varepsilon,\delta}
 -\overline{v}_{\varepsilon,\delta}\cdot\nabla_H \tilde{v}_{\varepsilon,\delta}
 -\tilde{u}_{\varepsilon,\delta}\cdot\nabla \tilde{v}_{\varepsilon,\delta}
 + \tfrac{1}{2}\int_{-1}^{+1}\tilde{u}_{\varepsilon,\delta}\cdot\nabla \tilde{v}_{\varepsilon,\delta}, \\
 \tilde{F}_2(\overline{v}_{\varepsilon,\delta}, \tilde{u}_{\varepsilon,\delta})&:=
 -\overline{v}_{\varepsilon,\delta}\cdot \nabla_H w_{\varepsilon,\delta}
 -\tilde{u}_{\varepsilon,\delta}\cdot \nabla w_{\varepsilon,\delta}.
 \end{align*}
A similar splitting holds for the primitive equations, compare e.g. \cite{CaoTiti2007} where it plays an important role in the derivation of global \textit{a priori} bounds.

\subsection{Anisotropic $3D$ Stokes and anisotropic $2D$ Navier-Stokes equations}
The system \eqref{eq:NS_eps2} can be compared to the following decoupled $2D$-Navier-Stokes  and average free $3D$-Stokes equations, i.e., for functions $\overline{v}_{0,\infty}, \overline{p}_{0,\infty} \tilde{v}_{0,\infty}^{\varepsilon,\delta}, w_{0,\infty}^{\varepsilon,\delta}$ and $\tilde{p}_{0,\infty}^{\varepsilon,\delta}$ with $\tilde{u}_{0,\infty}^{\varepsilon,\delta}=(\tilde{v}_{0,\infty}^{\varepsilon,\delta},w_{0,\infty}^{\varepsilon,\delta})$
\begin{equation}\label{eq:NS_eps3}
\left \{\begin{array}{rll}
\partial _t\overline{v}_{0,\infty}+\overline{v}_{0,\infty}\cdot\nabla \overline{v}_{0,\infty}-\Delta_{H} \overline{v}_{0,\infty}+\nabla _H\overline{p}_{0,\infty}=& \ 0 &\text{ in }(0,T)\times G ,\\
\partial _t\tilde{v}_{0,\infty}^{\varepsilon,\delta}-\Delta_{\delta} \tilde{v}_{0,\infty}^{\varepsilon,\delta}+\nabla _H\tilde{p}^{\varepsilon,\delta}_{0,\infty}=& \ 0
&\text{ in }(0,T)\times \Omega ,\\
\partial_tw_{0,\infty}^{\varepsilon,\delta}-\Delta_{\delta} w_{0,\infty}^{\varepsilon,\delta}+\tfrac{1}{\varepsilon ^2}\partial _z\tilde{p}^{\varepsilon,\delta}_{0,\infty}=& \ 
0
&\text{ in }(0,T)\times \Omega ,\\
\divergence \overline{v}_{0,\infty} =& \ 0&\text{ in }(0,T)\times  G ,\\
\divergence \tilde{u}^{\varepsilon,\delta}_{0,\infty} =& \ 0&\text{ in }(0,T)\times\Omega, \\
\tilde{u}_{0,\infty}^{\varepsilon,\delta} (0)=&\ (\tilde{u}_0)_{\varepsilon,\delta}&\text{ in }\Omega,\\
\overline{v}_{0,\infty} (0)=&\ (\overline{v}_0)_{\varepsilon,\delta}&\text{ in }G, \\
\int_{-1}^1\tilde{v}_{0,\infty}^{\varepsilon,\delta}=0, \quad \int_{-1}^1\tilde{p}_{0,\infty}^{\varepsilon,\delta}= & \ 0 &\text{ in }G, \\
\overline{p}_{0,\infty}, \overline{v}_{0,\infty}, \text{ periodic in }x,y,& &\\
\tilde{p}_{0,\infty}^{\varepsilon,\delta}, \tilde{v}_{0,\infty}^{\varepsilon,\delta}, w_{0,\infty}^{\varepsilon,\delta}\text{ periodic in }x,y,z,& &\\
\tilde{p}_{0,\infty}^{\varepsilon,\delta}, \tilde{v}_{0,\infty}^{\varepsilon,\delta}\text{ even in }z,& &\\
w_{0,\infty}^{\varepsilon,\delta}\text{ odd in }z.& &
\end{array}\right .\tag{NS$_{\varepsilon,\delta}^{(2D)}$}
\end{equation}

\subsection{Difference equation for $\varepsilon\to 0$ and $\delta\to \infty$}
Consider now the difference of solutions to \eqref{eq:NS_eps2} and \eqref{eq:NS_eps3} after multiplying the $w$-equations by $\varepsilon$. Let 
\begin{align}\label{eq:VW_2}
	\begin{split}
\bar{V}_{\varepsilon,\delta} &:=\bar{v}_{\varepsilon,\delta} -\bar{v}_{0,\infty}, \quad
\tilde{V}_{\varepsilon,\delta} :=\tilde{v}_{\varepsilon,\delta} -\tilde{v}_{0,\infty}^{\varepsilon,\delta}, \quad W_{\varepsilon,\delta}:=\varepsilon(w_{\varepsilon,\delta} -w_{0,\infty}^{\varepsilon,\delta}), \quad \hbox{and} \\
\tilde{U}_{\varepsilon,\delta} &:=(\tilde{V}_{\varepsilon,\delta},W_{\varepsilon,\delta}), \quad \overline{P}_{\varepsilon,\delta} :=\overline{p}_{\varepsilon,\delta} -\overline{p}_{0,\infty}, \quad\tilde{P}_{\varepsilon,\delta} :=\tilde{p}_{\varepsilon,\delta} -\tilde{p}^{\varepsilon,\delta}_{0,\infty},
\end{split}
\end{align}
and set
\begin{align}\label{eq:U_2}
U_{\varepsilon,\delta}:= (\overline{V}_{\varepsilon,\delta}+\tilde{V}_{\varepsilon,\delta}, W_{\varepsilon,\delta}),\quad
\tilde{u}_{0,\infty}^{\varepsilon,\delta}=(\tilde{v}_{0,\infty}^{\varepsilon,\delta}, w_{0,\infty}^{\varepsilon,\delta}),\quad \hbox{and}\quad
u_{0,\infty}^{\varepsilon,\delta}=(\overline{v}_{0,\infty}+\tilde{v}^{\varepsilon,\delta}_{0,\infty},  w^{\varepsilon,\delta}_{0,\infty}).
\end{align}
This solves
\begin{equation}\label{eq:NS_eps4}
\left \{\begin{array}{rll}
\partial _t\overline{V}_{\varepsilon,\delta}-\Delta_{H} \overline{V}_{\varepsilon,\delta}+\nabla _H\overline{P}_{\varepsilon,\delta}=& \overline{F}_{\varepsilon,\delta}^H&\text{ in }(0,T)\times G ,\\
\partial _t\tilde{V}_{\varepsilon,\delta}
-\Delta_{\delta} \tilde{V}_{\varepsilon,\delta}+\nabla _H\tilde{P}_{\varepsilon,\delta}=&\tilde{F}_{\varepsilon,\delta}^H &\text{ in }(0,T)\times\Omega,\\
\partial _tW_{\varepsilon,\delta}-\Delta_{\delta} W_{\varepsilon,\delta}+\tfrac{1}{\varepsilon }\partial _z\tilde{p}_{\varepsilon,\delta}=&\tilde{F}_{\varepsilon,\delta}^w
&\text{ in }(0,T)\times\Omega ,\\
\divergence_H \overline{V}_{\varepsilon,\delta} =& \ 0&\text{ in }(0,T)\times G ,\\
\divergence_{\varepsilon} \tilde{U}_{\varepsilon,\delta} =& \ 0&\text{ in }(0,T)\times\Omega, \\
\tilde{U}_{\varepsilon,\delta} (0)=&\ 0&\text{ in }\Omega,\\
\overline{V}_{\varepsilon,\delta} (0)=&\ 0&\text{ in }G, \\
\int_{-1}^1\tilde{V}_{\varepsilon,\delta}=0, \quad \int_{-1}^1\tilde{P}_{\varepsilon,\delta}= & \ 0 &\text{ in }G, \\
\overline{P}_{\varepsilon,\delta}, \overline{V}_{\varepsilon,\delta}, \text{ periodic in }x,y,& &\\
\tilde{P}_{\varepsilon,\delta}, \tilde{V}_{\varepsilon,\delta}, W_{\varepsilon,\delta}\text{ periodic in }x,y,z,& &\\
\tilde{P}_{\varepsilon,\delta}, \tilde{V}_{\varepsilon,\delta}\text{ even in }z,& &\\
W_{\varepsilon,\delta}\text{ odd in }z,& &
\end{array}\right .\tag{Diff$_{\varepsilon,\delta}^{(2)}$}
\end{equation}
where
\begin{align*}
\overline{F}_{\varepsilon,\delta}^H:=& \overline{F}(\tilde{u}_{\varepsilon,\delta}) - \overline{v}_{\varepsilon,\delta}\nabla_H\overline{v}_{\varepsilon,\delta}
+\overline{v}_{0,\infty}\nabla_H\overline{v}_{0,\infty}, \\
\tilde{F}_{\varepsilon,\delta}^H:=& \tilde{F}_1(\tilde{u}_{\varepsilon,\delta}, \overline{v}_{\varepsilon,\delta}) , \\
\tilde{F}_{\varepsilon,\delta}^w:=&\varepsilon\tilde{F}_2(\overline{v}_{\varepsilon,\delta}, \tilde{u}_{\varepsilon,\delta}).
\end{align*}
Inserting the relations from \eqref{eq:VW_2} this can be expressed as
\begin{align}\label{eq:F_Diff2}
	\begin{split}
\overline{F}_{\varepsilon,\delta}^H=&
-\overline{V}_{\varepsilon,\delta}\cdot\nabla_H \overline{V}_{\varepsilon,\delta}
-\overline{v}_{0,\infty}\cdot\nabla_H \overline{V}_{\varepsilon,\delta}
-\overline{V}_{\varepsilon,\delta}\cdot\nabla_H \overline{v}_{0,\infty}
\\ 
&-\tfrac{1}{2}\int_{-1}^{+1}(\tilde{V}_{\varepsilon,\delta}+\tilde{v}^{\varepsilon,\delta}_{0,\infty})\cdot\nabla_H (\tilde{V}_{\varepsilon,\delta}+\tilde{v}^{\varepsilon,\delta}_{0,\infty}) +
(\tfrac{1}{\varepsilon}W_{\varepsilon,\delta}+w^{\varepsilon,\delta}_{0,\infty})\cdot\partial_z (\tilde{V}_{\varepsilon,\delta}+\tilde{v}^{\varepsilon,\delta}_{0,\infty})
, \\
\tilde{F}_{\varepsilon,\delta}^H=&-(\tilde{V}_{\varepsilon,\delta}+\tilde{v}^{\varepsilon,\delta}_{0,\infty})\cdot\nabla_H (\overline{V}_{\varepsilon,\delta}+\overline{v}_{0,\infty})
-(\overline{V}_{\varepsilon,\delta}+\overline{v}_{0,\infty})\cdot\nabla_H (\tilde{V}_{\varepsilon,\delta}+\tilde{v}^{\varepsilon,\delta}_{0,\infty})
\\
&-(\tilde{V}_{\varepsilon,\delta}+\tilde{v}^{\varepsilon,\delta}_{0,\infty})\cdot\nabla_H (\tilde{V}_{\varepsilon,\delta}+\tilde{v}^{\varepsilon,\delta}_{0,\infty})
-(\tfrac{1}{\varepsilon}W_{\varepsilon,\delta}+w^{\varepsilon,\delta}_{0,\infty})\cdot\partial_z (\tilde{V}_{\varepsilon,\delta}+\tilde{v}^{\varepsilon,\delta}_{0,\infty})
\\
&+ \tfrac{1}{2}\int_{-1}^{+1}(\tilde{V}_{\varepsilon,\delta}+\tilde{v}^{\varepsilon,\delta}_{0,\infty})\cdot\nabla_H (\tilde{V}_{\varepsilon,\delta}+\tilde{v}^{\varepsilon,\delta}_{0,\infty})
+(\tfrac{1}{\varepsilon}W_{\varepsilon,\delta}+w^{\varepsilon,\delta}_{0,\infty})\cdot\partial_z (\tilde{V}_{\varepsilon,\delta}+\tilde{v}^{\varepsilon,\delta}_{0,\infty}),
\\
\tilde{F}_{\varepsilon,\delta}^w=&-(\overline{V}_{\varepsilon,\delta}+\overline{v}_{0,\infty})\cdot \nabla_H( W_{\varepsilon,\delta}+\varepsilon w^{\varepsilon,\delta}_{0,\infty})
\\
&-(\tilde{V}_{\varepsilon,\delta}+\tilde{v}^{\varepsilon,\delta}_{0,\infty})\cdot \nabla_H (W_{\varepsilon,\delta}+\varepsilon w^{\varepsilon,\delta}_{0,\infty}))
-(\tfrac{1}{\varepsilon}W_{\varepsilon,\delta}+ w^{\varepsilon,\delta}_{0,\infty})\cdot \partial_z (W_{\varepsilon,\delta}+\varepsilon w^{\varepsilon,\delta}_{0,\infty}).
\end{split}
\end{align}
Here again by \eqref{eq:Weps} the terms with $\tfrac{1}{\varepsilon}$ can be rewritten.
We will see that $\tilde{u}^{\varepsilon,\delta}_{0,\infty}\to 0$ in suitable norms as $\delta\to \infty$ uniformly for $\varepsilon\in(0,1]$ which will give a diminishing right hand side in \eqref{eq:NS_eps4}. 

\section{Function spaces}\label{sec:functionspaces}

\subsection{Isotropic function spaces on the torus}
For $\Omega=\Omega_1$, cf. \eqref{eq:Omega}, the periodic Bessel potential and Besov spaces will be needed. For $p,q\in (1,\infty )$ and $s\in [0,\infty )$ these are defined by
\begin{align*}
H^{s,p}_{per}(\Omega)&=\overline{C_{per}^\infty (\overline{\Omega })}^{\Vert \cdot\Vert _{H^{s,p}}} \quad \hbox{and} \quad B^{s}_{p,q,per}(\Omega )=\overline{C_{per}^\infty (\overline{\Omega })}^{\Vert \cdot\Vert _{B^{s}_{p,q}}}, \\
H^{s,p}_{per}(G)&=\overline{C_{per}^\infty (\overline{G})}^{\Vert \cdot\Vert _{H^{s,p}}} \quad \hbox{and} \quad B^{s}_{p,q,per}(G )=\overline{C_{per}^\infty (\overline{G})}^{\Vert \cdot\Vert _{B^{s}_{p,q}}},
\end{align*}
respectively. 
Here,
$C_{per}^\infty (\overline{\Omega })$ and $C_{per}^\infty (\overline{G })$ denote the space of smooth functions that are periodic of any order (cf. \cite[Section 2]{Hieber2016}) in all directions. The spaces $H^{s,p}$ denote the Bessel potential spaces of order $s$, with norm $\Vert \cdot \Vert _{H^{s,p}}$ defined via the restriction of the corresponding space defined on the whole space (cf. \cite[Definition 3.2.2.]{Triebel}). Analogously, $B^{s}_{p,q}$ denote the Besov spaces  which are defined by restrictions of functions on the whole space, see e.g. \cite[Definition 3.2.2.]{Triebel}. Note that $L^p=H^{0,p}_{per}$, 
and one sets $H^s:=H^{s,2}$. Periodic Bessel potential and Besov spaces can equivalently be defined via Fourier series, cf. \cite[Chapter 9]{Triebel}.

The divergence free conditions in the above sets of equations can be encoded into  spaces of solenoidal functions for $p\in (1,\infty)$ where we include the parity conditions
\begin{align*}
L^p_{\sigma, \varepsilon} (\Omega )&=\overline{\{u=(v,w)\in C_{per}^\infty (\overline{\Omega })^3\colon\divergence_H v + \tfrac{1}{\varepsilon}\partial_zw =0, \hbox{$v$ even in $z$ and $w$ odd in $z$}\}}^{\Vert \cdot\Vert _{L^p}}, \quad \varepsilon>0,\\
L^p_{\overline{\sigma }} (\Omega )&=\overline{\{v\in C_{per}^\infty (\overline{\Omega })^2:\divergence _H\overline{v} =0, \hbox{$v$ even in $z$}\}}^{\Vert \cdot\Vert _{L^p}}, \\
L^p_{\sigma} (G )&=\overline{\{\overline{v}\in C_{per}^\infty (\overline{G })^2:\divergence _H\overline{v} =0\}}^{\Vert \cdot\Vert _{L^p}}, 
\end{align*}
setting for brevity  $L^p_{\sigma} (\Omega ):=L^p_{\sigma, 1} (\Omega )$. Consider also the scaled Helmholtz projection 
\begin{align*}
\IP_{\varepsilon}\colon \{u=(v,w)\in L^p(\Omega)^3\colon\hbox{$v$ even in $z$ and $w$ odd in $z$}\} \rightarrow L^p_{\sigma, \varepsilon} (\Omega ), \quad \varepsilon>0,
\end{align*}
where we set $\IP:=\IP_{1}$, the hydrostatic and the $2D$-Helmholtz projections
\begin{align*}
\IP_{\overline{\sigma}}&\colon \{v\in L^p(\Omega)^2\colon\hbox{$v$ even in $z$ and $w$ odd in $z$}\} \rightarrow L^p_{\overline{\sigma}} (\Omega ), \\
\IP_{G}&\colon  L^p(G)^2 \rightarrow L^p_{\sigma} (G ).
\end{align*}


\subsection{Maximal $L^2_t$-$L^2_x$-regularity spaces}
We set
\begin{eqnarray*}
X_0 :=L^2(\Omega ), & &X_1:=H^{2}_{per}(\Omega ),\\
X^v_0:=L^2_{\overline{\sigma}}(\Omega ), & &X^v_1:=H^{2}_{per}(\Omega )^2 \cap L^2_{\overline{\sigma}}(\Omega ), \\
X_{0,\varepsilon}^u:= L^2_{\sigma,\varepsilon} (\Omega ),
& &X_{1,\varepsilon}^u:= H^{2}_{per}(\Omega)^3\cap L^2_{\sigma,\varepsilon} (\Omega ), \quad \varepsilon>0
\end{eqnarray*}
where we use in the case $\varepsilon=1$ the short hand notations $X_{0}^u$ and $X_{1}^u$.
For 
$t, T\in [0, \infty]$ we also define the maximal regularity spaces 
\[
\IE _0(t, T) := L^2(t,T;X_0),\quad \IE _1(t, T) := L^2(t,T;X_1)\cap H^{1}(t,T; X_0),
\]
and analogously $\IE _i^v(t, T)$ and $\IE _i^u(t, T)$ with respect to $X_i^v$ and $X_i^u$, respectively ($i = 0, 1$).
To describe the regularity of the pressure, we introduce 
\begin{align*}
\IE_1^p(t, T) := \{ p \in L^2(t, T; H_{per}^{1}(\Omega)) \colon \int_\Omega p(s) \, dx = 0 \text{ for a.e.\ } s \in (t, T), \quad p \hbox{ even in }z\}.
\end{align*}
To simplify our notation, we sometimes write only $\IE_i(t, T)$ without superscripts and $\IE_i(T)$ when $t = 0$. 
To consider $\delta$-dependent norms, we set 
\begin{align*}
\IE_{H,\delta}(t,T):=\IE_1(t,T) \quad \hbox{with} \quad \norm{u}_{\IE_{H,\delta}}:=  \norm{u}_{\IE_{0}}+\norm{\partial_t u}_{\IE_{0}}+\norm{\Delta_\delta u}_{\IE_{0}}, \quad \delta>0.
\end{align*}


Then, consider the traces spaces given by  the real interpolation functor $(\cdot,\cdot)_{\theta,q}$  as
\begin{align*}
X_\gamma :=(X_0,X_1)_{1/2,2},\quad
X^u_\gamma :=(X_{0,\varepsilon}^u,X_{1,\varepsilon}^u)_{1/2,2},\quad \hbox{and} \quad X^v_\gamma :=(X_0^v,X_1^v)_{1/2,2},
\end{align*}
where when there is no ambiguity we use the short hand notation $X_\gamma$ for each of the three.  
Following the lines of  \cite[Section 4]{HieberKashiwabaraHussein2016} and \cite{GigaGriesHieberHusseinKashiwabara2017} the trace or initial value spaces
can be characterized as follows, where one uses that $B^s_{2,2}=H^s$.
\begin{lemma}[Characterization of the initial value spaces]
	Let $p,q\in (1,\infty )$ and $\varepsilon>0$. Then 
	\begin{align*}
		X_\gamma =
 H^{1}_{per}(\Omega ), \quad
		X^u_\gamma =
H^{1}_{per}(\Omega )^3\cap L^2_{\sigma,\varepsilon} (\Omega ), \quad \hbox{and} \quad 
		X^v_\gamma =
 H^{1}_{per}(\Omega )^2\cap L^2_{\overline{\sigma}} (\Omega ).
	\end{align*}
\end{lemma}


\subsection{Anisotropic function spaces}\label{subsec:aniiso}

The anisotropic structure of the
primitive equations motivates the definition of the anisotropic Bessel potential spaces for $r,s\geq 0$ and $q,p\in [1,\infty]$
\begin{align*}
H^{r,q}_{z}H^{s,p}_{xy}:=H^{r,q}(-1,1;H^{s,p}_{per}(G)), \hbox{ where }  H^{s,p}_{xy}:=H^{s,p}_{per}(G), \quad H^{r,q}_{z}:=H^{r,q}(-1,1),
\end{align*} 
with the norms
\begin{align*}
\norm{v}_{H^{r,q}_{z}H^{s,p}_{xy}} = \big\Vert\norm{v(\cdot,\cdot,z)}_{H_{xy}^{s,p}} \big\Vert_{H^{r,q}_{z}}. 
\end{align*}
For these spaces  
there is an
anisotropic Hölder's inequality, that is,  for $\tfrac{1}{p}=\tfrac{1}{p_1}+\tfrac{1}{p_2}$ and
$\tfrac{1}{q}=\tfrac{1}{q_1}+\tfrac{1}{q_2}$ 
\begin{align}\label{eq:Holder}
\norm{v u}_{L^{q}_{z}L^{p}_{xy}} 
\leq  
\big\Vert\norm{v(\cdot,\cdot,z)}_{L_{xy}^{p_1}}\norm{u(\cdot,\cdot,z)}_{L_{xy}^{p_2}} \big\Vert_{L^{r,q}_{z}} = 
\norm{v}_{L^{q_1}_{z}L^{p_1}_{xy}}\norm{u}_{L^{q_2}_{z}L^{p_2}_{xy}}.
\end{align}
Also the two-dimensional  Ladzhenskaya inequality carries over to 
\begin{align}\label{eq:Ladyzhenskaya}
\norm{v}_{L^{2}_{z}L^{4}_{xy}}^2 
\leq  
\int_{-1}^1\norm{v(\cdot,\cdot,z)}_{L_{xy}^{2}}\norm{v(\cdot,\cdot,z)}_{H^1_{xy}} dz 
\leq 
\norm{v}_{L^{2}_{z}L^{2}_{xy}}\norm{v}_{L^{2}_{z}H^{1}_{xy}}.
\end{align}
To account for the anisotropic time-space setting, consider 
\begin{align*}
\IE_H(t,T)&:=L^2(t,T;L_z^2H_{xy}^2) \cap  H^1(0,T;L^2(\Omega)), \\
\IE_z(t,T)&:= L^2(t,T;H_z^1H_{xy}^1)\cap L^\infty(0,T;H_z^1L_{xy}^2), \\
\IE_{z,H}(t,T)&:= \IE_{z}(t,T)\cap \IE_{H}(t,T)
\end{align*}
with $\delta$-independent norms. 
For $\delta>0$ set $\IE_{H,\delta}(t,T):=\IE_{1}(t,T)$ with $\delta$-dependent norm 
\begin{align*}
\norm{u}_{\IE_{H,\delta}(t,T)}:= \norm{u}_{L^2(t,T;L^2(\Omega))} + 
\norm{\partial_t u}_{L^2(t,T;L^2(\Omega))}
+\norm{\Delta_\delta u}_{L^2(t,T;L^2(\Omega))}.
\end{align*}
Note that 
\begin{align}\label{eq:HdembeddingH}
\norm{u}_{\IE_{H}(t,T)} \leq \norm{u}_{\IE_{H,\delta}(t,T)}\quad\hbox{for all } \delta>0 \hbox{ and } u\in \IE_{H,\delta}(t,T).
\end{align}
Motivated by the anisotropic diffusion and transport behaviour of the primitive equations with horizontal viscosity, cf.  \cite{CaoLiTiti_PHE_2017, HusseinSaalWrona},  and one sets for $\eta>0$ 
\begin{align}\label{eq:H1eta}
\begin{split}
H^{1}_{\bar\sigma, \eta}(\Omega)&:= \{ v\in H^1(\Omega)^2 \cap L^2_{\bar\sigma}(\Omega)\cap L^\infty(\Omega)\colon  \norm{v}_{H^1} + \norm{v}_{L^\infty}+ \norm{\partial_z v}_{L^{2+\eta}}<\infty \} \quad \hbox{with} \\
 \norm{v}_{H^1_{\bar\sigma, \eta}}&:=  \norm{v}_{H^1} + \norm{v}_{L^\infty}+ \norm{\partial_z v}_{L^{2+\eta}}.
 \end{split}
\end{align}


\section{Main results}\label{sec:main}
In the following we state our main results on  strong solutions and their convergence under the hydrostatic approximation. 
We say that 
$u_{\varepsilon,\delta}$ is a \emph{strong solution to the scaled anisotropic Navier-Stokes equations} \eqref{eq:NS_eps} for $\varepsilon,\delta>0$  on $(0,T)$ in the $L^2$-$L^2$-setting, if $u_{\varepsilon,\delta}\in \IE _1^u(T)$ and there exists $p_{\varepsilon,\delta}\in \IE_1^p(T)$ such that \eqref{eq:NS_eps} holds almost everywhere.
\begin{proposition}[Local existence for \eqref{eq:NS_eps}]\label{prop:NS_eps}
	Let $T\in (0,\infty)$,
	\begin{align*}
\varepsilon,\delta>0,\quad \hbox{and}\quad  (u_{\varepsilon,\delta})_0\in 
H^{1}_{per}(\Omega )^3\cap L^2_{\sigma} (\Omega ). 
	\end{align*}
	Then there exists a maximal existence time $T_{\varepsilon,\delta}=T_{\varepsilon,\delta}((u_{\varepsilon,\delta})_0)\in (0,T]$ such that there exits a unique  strong solution $u_{\varepsilon,\delta}\in \IE^u_{1}(T_{\varepsilon,\delta})$ to~\eqref{eq:NS_eps}.
\end{proposition}
 The proof of this local well-posedness result follows e.g. with Lemma~\ref{lemma:NavierStokes_nonlinear}  below. Of course, there are much more refined results on the Navier-Stokes equations in $L^q_t$-$L^p_x$-spaces, cf. e.g. \cite{MR3808592, MR3721420, MR3722064}, however, Proposition~\ref{prop:NS_eps} is sufficient  for our  comparison arguments.

We say that $u=(v,w)$ is a \emph{strong solution to the primitive equations with horizontal viscosity} \eqref{eq_PEH} on $(0,T)$ if there is $v\in \IE_z(T)\cap \IE_H(T)$ and there exists $p \in \IE_1^p(T)$ such that \eqref{eq_PEH} holds almost everywhere. 
The existence of such global strong solutions to \eqref{eq_PEH} has been proven by Li and Titi in \cite{CaoLiTiti_PHE_2017}, compare also \cite{HusseinSaalWrona} for a slightly different setting. The higher regularity of these solutions and the proofs of the following results are discussed in Section~\ref{sec:additional_reg}.


\begin{proposition}[Global existence for \eqref{eq_PEH} and higher regularity]\label{prop:PE_H}
Let $T\in (0,\infty)$ and $\eta>0$.
\begin{enumerate}[(a)]
	\item If $v_0\in H_{\bar\sigma, \eta}^1(\Omega)$, then there exists a unique strong solution to \eqref{eq_PEH} with
	\begin{align*}
	v\in \IE_H(T) \cap \IE_z(T).
	\end{align*}
	\item If in addition to the conditions in (a) one has $\partial_z^2 v_0\in L^2(\Omega)^2$ and $w_0=w(v_0)\in H^1(\Omega)$, then 
	\begin{align*}
	v\in \IE_1^v(T) \cap L^\infty(0,T;H^2_zL^2_{xy}) \quad \hbox{and} \quad w(v)\in \IE_1(T).
	\end{align*}
	\item If in addition to the conditions in (a) one has $\partial_z v_0\in H^1(\Omega)^2$,  then 
	\begin{align*}
v\in L^4(0,T; H_z^1H_{xy}^{1,4}).
	\end{align*}  
		\item If in addition to the conditions in (a) one has  $\partial_z^3 v_0\in L^2(\Omega)^2$, then 
	\begin{align*}
	v\in L^\infty(0,T;H^3_zL^2_{xy}) \cap L^2(0,T;H^3_zH^1_{xy}) .
	\end{align*}  
\end{enumerate}
\end{proposition}
Note that $v_0, \partial_z v_0\in H^1(\Omega)^2\cap L^\infty(\Omega)$ implies already $v_0\in H_{\bar\sigma, \eta}^1(\Omega)$ for some $\eta>0$, cf. \eqref{eq:H1eta} for the definition of this space.   
Considering the difference \eqref{eq_Diff} between \eqref{eq:NS_eps} and \eqref{eq_PEH}, we can state our first main result the proof of which is given in Section~\ref{sec:quadratic_inequality}.
\begin{theorem}[Convergence for $\varepsilon\to 0$ and $\delta \to 0$]\label{thm:main1}
Let $T\in (0,\infty)$ and
\begin{align*}
u_0=(v_0,w_0)\in H^1(\Omega)\cap L^2_\sigma(\Omega) \quad \hbox{with}\quad v_0\in L^\infty(\Omega),\quad \partial_z v_0\in H^1(\Omega)^2, \quad \partial_z^2v_0, \partial_z^3v_0\in L^2(\Omega)^2. 
\end{align*}
\begin{enumerate}[(1)]
	\item Let $u=(v,w)$ be the solution to \eqref{eq_PEH}  referred to in Proposition~\ref{prop:PE_H}, and 
	\item for $\varepsilon,\delta>0$ let
	$u_{\varepsilon,\delta}$  be the solution to \eqref{eq:NS_eps} referred to in Proposition~\ref{prop:NS_eps}. 
\end{enumerate}
Then,
\begin{enumerate}[(a)]
	\item there exists a constant $c>0$ such that  one has for the existence time $T_{\varepsilon,\delta}$ of $u_{\varepsilon,\delta}$
	\begin{align*}
	T_{\varepsilon,\delta}=T \quad \hbox{for}\quad \varepsilon, \delta <c;
	\end{align*}
	\item there exist constants $C,c>0$ independent of $\varepsilon, \delta$ such that for $\varepsilon, \delta<c$
	\begin{align*}
\norm{(v_{\varepsilon,\delta} -v,\varepsilon(w_{\varepsilon,\delta} -w))}_{\IE_{H,\delta}(T)}
+\norm{(v_{\varepsilon,\delta} -v,\varepsilon(w_{\varepsilon,\delta} -w))}_{\IE_{z}(T)} \leq C(\varepsilon+\delta).
\end{align*}
\end{enumerate} 
\end{theorem}
In particular from (b) and \eqref{eq:HdembeddingH} convergence in $\IE_{H}(T)\cap \IE_z(T)$ follows as $\varepsilon\to 0$ and $\delta\to 0$. Recall that these norms have been introduced in Subsection~\ref{subsec:aniiso}.  Part (a) implies global strong well-posedness of \eqref{eq:NS_eps} for $\varepsilon,\delta$ sufficiently small.  

\begin{remark}[Comparison to the results in \cite{LiTitiGuozhi_2022} by Li, Titi and Yuan]\label{rem:comp}
	The result obtained in \cite{LiTitiGuozhi_2022} deals with the case $\nu_z = \varepsilon^\gamma$ for $\gamma>2$ and convergence in the weak and strong sense in $\varepsilon$-dependent energy norms. For $\delta=\delta_{\varepsilon}:=\varepsilon^{\gamma-2}$, they estimate in \cite[Theorem 1.1]{LiTitiGuozhi_2022} and \cite[Theorem 1.2]{LiTitiGuozhi_2022}
	\begin{align*}
		\norm{(V_{\varepsilon,\delta}, W_{\varepsilon,\delta})}_{L^\infty(0,T;L^2(\Omega))}
		+
		\norm{\nabla_H(V_{\varepsilon,\delta}, W_{\varepsilon,\delta})}_{L^2(0,T;L^2(\Omega))}
		+
		\delta_{\varepsilon}^{1/2}\norm{\partial_z(V_{\varepsilon,\delta}, W_{\varepsilon,\delta})}_{L^2(0,T;L^2(\Omega))} &= \mathcal{O}(\varepsilon + \delta_{\varepsilon}^{1/2}), \\
		\norm{(V_{\varepsilon,\delta}, W_{\varepsilon,\delta}}_{L^\infty(0,T;H^1(\Omega)))}
		+
		\norm{\nabla_H(V_{\varepsilon,\delta}, W_{\varepsilon,\delta})}_{L^2(0,T;H^1(\Omega))}
		+
		\delta_{\varepsilon}^{1/2}\norm{\partial_z(V_{\varepsilon,\delta}, W_{\varepsilon,\delta})}_{L^2(0,T;H^1(\Omega))} &= \mathcal{O}(\varepsilon + \delta_{\varepsilon}^{1/2}),
	\end{align*}	
	respectively,
	where we have translated their result to our convention \eqref{eq:VW}.
The convergence rate is in both cases of order
	\begin{align*}
	\mathcal{O}(\varepsilon + \delta_{\varepsilon}^{1/2})=	\mathcal{O}(\varepsilon^{\beta_1})\quad \hbox{with}\quad \beta_1 = \min \{\gamma/2-1, 1\}.
	\end{align*}
	Here, we estimate slightly differently scaled norms  where we have orders $\delta$ instead of $\delta^{1/2}$ on both the left and the right hand side. More concretely, Theorem~\ref{thm:main1} says that 
	\begin{multline*}
	\norm{(V_{\varepsilon,\delta}, W_{\varepsilon,\delta})}_{H^1(0,T;L^2(\Omega))}
	+
	\norm{\Delta_H(V_{\varepsilon,\delta}, W_{\varepsilon,\delta})}_{L^2(0,T;L^2(\Omega))}
	+
	\delta\norm{\partial_z^2(V_{\varepsilon,\delta},W_{\varepsilon,\delta})}_{L^2(0,T;L^2(\Omega))} \\
	+
		\norm{\partial_z(V_{\varepsilon,\delta}, W_{\varepsilon,\delta})}_{L^\infty(0,T;L^2(\Omega))}
		+
		\norm{\partial_z\nabla_H(V_{\varepsilon,\delta}, W_{\varepsilon,\delta})}_{L^\infty(0,T;L^2(\Omega))}
		= \mathcal{O}(\varepsilon + \delta),
\end{multline*}	
where the  convergence rate in Theorem~\ref{thm:main1} with $\delta =\delta_{\varepsilon}= \varepsilon^{\gamma-2}$ becomes
	\begin{align*}
		\mathcal{O}(\varepsilon+\delta_{\varepsilon})=\mathcal{O}(\varepsilon^{\beta_2}) \quad \hbox{with}\quad \beta_2 = \min \{\gamma-2, 1\}\geq \beta_1.
	\end{align*}
	On the one hand, this means that the convergence rate improves here slightly in particular for  the terms $\norm{(V_{\varepsilon,\delta}, W_{\varepsilon,\delta})}_{H^1(0,T;L^2(\Omega))}$
and
	$\norm{\Delta_H(V_{\varepsilon,\delta}, W_{\varepsilon,\delta})}_{L^2(0,T;L^2(\Omega))}$. Here, more  generally   we include   the limit $\delta\to 0$ independent of $\varepsilon$. On the other hand, we 	 prove only convergence in the strong sense, and we have to impose stronger regularity assumptions on the initial data as compared to $v_0\in H^1$  and $v_0\in H^2$ with $\partial_z v_0\in L^p(\Omega)$ for $p>2$ in Theorem 1.1 and Theorem 1.2 in \cite{LiTitiGuozhi_2022}, respectively.
	
		The proof 
	in \cite{LiTitiGuozhi_2022} relies on sophisticated energy estimates and the factor $\delta_{\varepsilon}^{1/2}$ in the scaled norms seems to originate from testing with unscaled functions. In contrast  
 the proof  given here is based on quadratic norm inequalities discussed  in  Section~\ref{sec:quadratic_inequality}, and therefore the factor $\delta$ in the scaled norms appears naturally considering the difference equation~\eqref{eq_Diff}.  
	The results in \cite{LiTitiGuozhi_2022}
	and the ones presented here have been obtained independently, and parts of the results given here are included in \cite[Chapter 5]{Wrona_PhD}.  	
\end{remark}

We say that
$\overline{v}_{0,\infty}$ is a \emph{strong solution to the $2D$-Navier-Stokes equations} and 
$\tilde{u}_{0,\infty}^{\varepsilon,\delta}$ a solution of the scaled Stokes equations
on $(0,T)$ in the $L^2$-$L^2$-setting, if $\overline{v}_{0,\infty}\in \IE _1(T)$, $\tilde{u}^{\varepsilon,\delta}_{0,\infty}\in \IE _1(T)$, and there exists $\overline{p}_{0,\infty}\in \IE_1^p(T)$ 
and $\tilde{p}_{0,\infty}^{\varepsilon,\delta}\in \IE_1^p(T)$
such that \eqref{eq:NS_eps3} holds almost everywhere. 

\begin{proposition}[Global existence for the $2D$ Navier-Stokes equations]\label{prop:NS_H}
	Let $T\in (0,\infty)$, $\varepsilon,\delta>0$
	\begin{align*}
\overline{v}_0 \in H^1(G)^2\cap L_{\sigma}^2(G)   
	 \quad \hbox{and}\quad 
	 \tilde{u}_0 \in H^1(\Omega) \cap L^2_{\sigma}(\Omega) \quad \hbox{with}\quad \int_{-1}^1 \tilde{u}_0(\cdot,\cdot, \xi) d\xi=0.
	\end{align*} 
	Then  there exists a unique strong solution to \eqref{eq:NS_eps3}
	\begin{align*}
\overline{v}_{0,\infty}\in \IE_1(T)  \quad \hbox{and}\quad \tilde{u}^{\varepsilon,\delta}_{0,\infty}\in \IE_1(T). 
	\end{align*} 
\end{proposition}

\begin{theorem}[Convergence for $0<\varepsilon\leq 1$ and $\delta \to \infty$]\label{thm:main2}
	Let $T\in (0,\infty)$ and  $\varepsilon,\delta>0$
	\begin{align*}
	\overline{v}_0 \in H^1(G)^2\cap L_{\sigma}^2(G) \quad \hbox{and} \quad \tilde{u}_0\in H^{3/2}(\Omega)\cap L^2_{\sigma}(\Omega) \quad \hbox{with}\quad \int_{-1}^1 \tilde{u}_0(\cdot,\cdot, \xi) d\xi=0.
	\end{align*} 
\begin{enumerate}[(1)]
	\item  Let $\overline{v}_{0,\infty}$,$\tilde{u}^{\varepsilon,\delta}_{0,\infty}$ be the solution to \eqref{eq:NS_eps3} referred to in Proposition~\ref{prop:NS_H};
	\item  let
	$u_{\varepsilon,\delta}$  be the solution to \eqref{eq:NS_eps} referred to in Proposition~\ref{prop:NS_eps} which decomposes into 
	$u_{\varepsilon,\delta}=\overline{v}_{\varepsilon,\delta}+\tilde{u}_{\varepsilon,\delta}$.
\end{enumerate}
Then,
\begin{enumerate}[(a)]
	\item there exists a constant $c>0$ such that  one has for the existence time $T_{\varepsilon,\delta}$ of $u_{\varepsilon,\delta}$
	\begin{align*}
	T_{\varepsilon,\delta}=T \quad \hbox{for}\quad \varepsilon, \tfrac{1}{\delta} <c;
	\end{align*}
	\item there exist constants $C,c>0$ independent of $\varepsilon, \delta$ such that for all $\delta>0$ with $\tfrac{1}{\delta}<c$
	\begin{align*}
\sup_{\varepsilon\in (0,1] }\left(\norm{\overline{v}_{0,\infty}-\overline{v}_{\varepsilon,\delta}}_{\IE_1} +\norm{\tilde{u}_{\varepsilon,\delta}}_{L^4(0,T;H^{3/2}(\Omega))}\right) \leq \frac{C}{\delta^{1/4}}.
\end{align*} 
\end{enumerate} 
\end{theorem}
This implies in particular the convergence of $\overline{v}_{\varepsilon,\delta}$ to $\overline{v}_{0,\infty}$ in $\IE_1(T)$ as $\delta \to \infty$ uniformly in $\varepsilon\in (0,1]$.

\begin{remark}[Further convergences]\label{rem:convergence}
	Theorem~\ref{thm:main1} discusses the case $\varepsilon \to 0$ and $\delta \to 0$, Theorem~\ref{thm:main2}  deals with the case when  $\varepsilon \in(0,1]$ and $\delta \to \infty$. Further cases are depicted in Figure~\ref{fig:epsdelta} and behave as follows:
	\begin{enumerate}
		\item \label{item:conv_1} In the case when $\varepsilon\to 0$ and $\delta>0$ is constant the solution of \eqref{eq:NS_eps} converges to the solution of the anisotropic primitive equations \eqref{eq_PE} with rate $\mathcal{O}(\varepsilon)$. The sense and the norm of the convergence are given in \cite{MR3926040} for $L^2$-$L^2$-spaces, in \cite{convergence} for $L^p$-$L^q$-spaces for certain $p,q\in (1,\infty)$ and in \cite{convergence2} for Dirichlet boundary conditions.
		\item In the case when $\varepsilon>0$ is constant and $\delta\to 0$ the solution of \eqref{eq:NS_eps} converges to the solution of the Navier-Stokes equations with only horizontal viscosity, compare e.g. \cite{Danchin_book} for the full space case, where the arguments there can be transferred to the torus. 
		\item The primitive equations with only horizontal viscosity are also obtained from the anisotropic primitive equations when $\delta \to 0$. This limit is proven in \cite{CaoLiTiti_PHE_2017} as a part of the analysis for the primitive equations with horizontal viscosity.
			\item The convergence of the anisotropic primitive equations \eqref{eq_PE} for $\delta \to \infty$ to the $2D$-Navier-Stokes equations follows from Theorem~\ref{thm:main2} since
				\begin{multline*}
			\lim_{\varepsilon\to 0} \left(\norm{\overline{v}_{0,\infty}-\overline{v}_{\varepsilon,\delta}}_{\IE_1} +\norm{\tilde{u}_{\varepsilon,\delta}}_{L^4(0,T;H^{3/2}(\Omega))}\right) 
			=
			\left(\norm{\overline{v}_{0,\infty}-\overline{v}_{0,\delta}}_{\IE_1} +\norm{\tilde{u}_{0,\delta}}_{L^4(0,T;H^{3/2}(\Omega))}\right)
			\leq \frac{C}{\delta^{1/4}}.
			\end{multline*} 
			Here $u_{0,\delta}=\overline{v}_{0,\infty}+\tilde{u}_{0,\delta}$ solves the anisotropic primitive equations \eqref{eq_PE} for $\delta >0$ where one used  for $\varepsilon\to 0$ the convergence from part \eqref{item:conv_1} of this remark.
	\end{enumerate}	
Thus, all the convergences from Figure~\ref{fig:epsdelta} can be made rigorous. 
\end{remark}

\begin{remark}[Boundary conditions]
	The Neumann boundary conditions on top and bottom  chosen here are essential for the proof of Theorem~\ref{thm:main1}, because it allows one to reduce the problem by symmetries to a periodic setting and thereby to avoid boundaries in this direction at all. A similar setting has been considered in the case of the primitive equations with full viscosity in \cite{MR3926040}  and \cite{convergence}, while in \cite{convergence2} Dirichlet boundary conditions lead to technical challenges which in the end can be overcome to give the analogous convergence result.
	
	Here in Theorem~\ref{thm:main1} however, the limit equation, that is, the primitive equations with only horizontal viscosity, is well-posed even without any boundary conditions in vertical direction, see e.g. \cite{HusseinSaalWrona}. This contrasts with the limiting equations, that is the rescaled Navier-Stokes equations, and therefore a boundary layer formation can be expected similar to the case of the limit from Navier-Stokes to Euler equations. 
	
	In Theorem~\ref{thm:main2}, this problem does not occur since the comparison takes place essentially in the horizontal variables, and therefore also other boundary conditions in the limiting equations could be included.   
\end{remark}


\section{Anisotropic linear estimates}\label{sec:lin}

\subsection{Uniform maximal regularity estimates  for the scaled Stokes equations}
The linearization of the difference equation \eqref{eq_Diff}
around the zero solution are the anisotropic scaled Stokes equations
\begin{equation}\label{eq:scaledStokes}
\left \{\begin{array}{rll}
\partial _tU- \Delta_{\delta} U +  \nabla_{\varepsilon}P=&  F&\text{ in }(0,T)\times\Omega \\ 
\divergence_{\varepsilon} U =& \ 0&\text{ in }(0,T)\times\Omega ,\\
U(0)=&\ U_0&\text{ in }\Omega , \\
P \text{ periodic in }x,y,z & \text{ even }&\text{ in }z,
\\
V ,W \text{ periodic in }x,y,z,&\text{ even and odd }&\text{ in }z, \text{ respectively.}\\
\end{array}\right .
\end{equation}

\begin{proposition}
	\label{prop:maxreg} Let $T>0$ and $\varepsilon,\delta>0$, then there exist a constant $C>0$ 
	independent of $\varepsilon$ and $\delta$ such that
	for all $\varepsilon,\delta>0$, and for all
	\begin{align*}
	U_0\in \X_\gamma^u \quad\hbox{and} \quad F\in \IE_0^u(T)
	\end{align*}
	there is a unique solution $U\in \IE_1^u(T)$ and $P\in \IE_1^p(T)$ to \eqref{eq:scaledStokes}  satisfying
	\begin{align*}
	\norm{\partial_t U}_{\IE_0} + \norm{\Delta_{\delta} U}_{\IE_0} + \norm{\nabla_\varepsilon P}_{\IE_0}  \leq C( \norm{F}_{\IE_0} + \norm{U_0}_{X_\gamma^u}).
	\end{align*} 
\end{proposition}
\begin{proof}
The statement has been proven in \cite[Proposition 5.1]{convergence} for $\delta=1$, and the proof carries over to any $\delta>0$ since one estimates $\nabla_\varepsilon P$ by $F$. 
%
\end{proof}

\subsection{The scaled Stokes equation and the $2D$-Navier-Stokes equations}

\begin{proposition}\label{prop:2DNS}
Let $T\in (0,\infty)$, $\varepsilon, \delta>0$, and
\begin{align*}
(\overline{v}_0)_{0,\infty}\in H^1(G)^2\cap L^2_{\sigma}(G), \quad 
	(\tilde{u}_0)_{0,\infty}\in H^{3/2}(\Omega)^3\cap L^2_{\sigma}(\Omega) \quad \hbox{with}\quad \int_{-1}^1 (\tilde{u}_0)_{0,\infty}(\cdot,\cdot,z)dz=0.
\end{align*}
Then there exists a unique solution to \eqref{eq:NS_eps3}
\begin{align*}
	\overline{v}_{0,\infty}\in \IE_1(T)\quad \hbox{and} \quad  \tilde{u}_{0,\infty}^{\varepsilon,\delta}=(\tilde{v}^{\varepsilon,\delta}_{0,\infty}, w^{\varepsilon,\delta}_{0,\infty})\in \IE_1(T),
\end{align*}
and moreover
there is a constant $C>0$ independent of $\varepsilon$ and $\delta$ such that for $\delta\geq1$ 
\begin{align*}
	\sup_{\varepsilon\in (0,1]}\norm{(\tilde{v}_{0,\infty}^{\varepsilon,\delta}, \varepsilon w_{0,\infty}^{\varepsilon,\delta})}_{L^4(0,T;H^{{3/2}}(\Omega))} \leq \frac{C}{\delta^{1/4}} \norm{((\tilde{v}_0)_{0,\infty}, (w_0)_{0,\infty})}_{H^{3/2}(\Omega)}.
\end{align*}
\end{proposition}
\begin{proof}
	The global well-posedness of the $2D$-Navier-Stokes equations in maximal $L^2$-regularity spaces is well-known and not proven here. It follows by standard regularity theory from the global well-posedness in the Leray-Hopf class. 
	
	Concerning the scaled Stokes equations, we consider \eqref{eq:scaledStokes} with 
	\begin{align*}
	F=0\quad \hbox{and} \quad	(U_0)_\varepsilon = 	((\tilde{v}_0)_{0,\infty}, \varepsilon(w_0)_{0,\infty})\quad \hbox{for}\quad \varepsilon \in (0,1].
	\end{align*}
	Using that $\Delta_H$ and $\partial_z^2$ are resolvent commuting, and that $\Delta_\delta$ and $\IP_\varepsilon$ commute, one has that
	\begin{align*}
	U_{\varepsilon,\delta}(t)=  e^{\IP_{\varepsilon}\Delta_\delta t}(U_{0})_\varepsilon
		=	 e^{\Delta_\delta t}(U_{0})_\varepsilon
		= (e^{\Delta_\delta t}(\tilde{v}_0)_{0,\infty}, \varepsilon e^{\Delta_\delta t}(w_0)_{0,\infty})=:(\tilde{v}_{0,\infty}^{\varepsilon,\delta}, \varepsilon w^{\varepsilon,\delta}_{0,\infty})
	\end{align*}
	solves \eqref{eq:scaledStokes}, and hence the existence of a unique solution in $\IE_1(T)$ to~\eqref{eq:NS_eps3} follows. 
	
		One has using that $\Delta_H$ and $\partial_z^2$ are resolvent commuting and that $\Delta$ and $\IP_\varepsilon$ commute that
	\begin{align*}
		(-\Delta)^{3/4} U_{\varepsilon,\delta}(t)= \Delta e^{\IP_{\varepsilon}\Delta_\delta t}(U_{0})_\varepsilon
		=	(-\Delta)^{3/4} e^{\delta\partial_z^2t}e^{\Delta_Ht}(U_{0})_\varepsilon
		= e^{(\delta-1)\partial_z^2t} e^{\Delta t}	(-\Delta)^{3/4}(U_{0})_\varepsilon 
	\end{align*}
	for $\delta\geq 1$ and  $t> 0$.
	Then using that $e^{\delta\partial_z^2t}$ is exponentially decaying on the vertically average free functions and that the decay bound $\omega>0$ is the spectral bound  for analytic semigroups, one estimates 
	for $p\in (1,\infty)$ and $T\in (0,\infty]$
	\begin{align*}
		\norm{	(-\Delta)^{3/4}U_{\varepsilon,\delta}}_{L^p(0,T;L^2(\Omega))} &=\norm{ 	 e^{(\delta-1)\partial_z^2t}e^{\Delta t}(-\Delta)^{3/4}(U_{0})_\varepsilon }_{L^p(0,T;L^2(\Omega))} \\
		&\leq C \norm{e^{-\omega (\delta-1) t}}_{L^{p}(0,T)}
		\norm{ e^{\partial_z^2t}e^{\Delta_H^2t}	(-\Delta)^{3/4}(U_{0})_\varepsilon }_{L^{\infty}(0,T;L^2(\Omega))} \\
		&\leq\tfrac{C}{(\omega(\delta-1)p)^{1/p}} \norm{(U_{0})_\varepsilon }_{H^{3/2}(\Omega)}.
	\end{align*}
	Taking the supremum over $\varepsilon\in (0,1]$ and $p=4$ concludes the proof using that $-\Delta$ is boundedly invertible on the periodic and vertically average free functions.  
\end{proof}

\subsection{Preservation of vertical regularity}\label{subsec:vertical_reg}
Instead of the scaled Stokes equations \eqref{eq:scaledStokes} one can also consider the following linearization of \eqref{eq_Diff} of diffusion-transport type. For given functions $\nu,\omega, f, U_0$,
consider
\begin{align*}
U=(U_1,U_2,U_3)\colon \Omega\rightarrow \IR^3, \quad P\colon \Omega \rightarrow \IR ,
\end{align*}
satisfying for fixed $\varepsilon,\delta>0$
\begin{equation}\label{eq:transport}
\left \{\begin{array}{rll}
\partial _tU- \Delta_{\delta} U + \partial_z (\omega U) + \partial_z (\tfrac{1}{\varepsilon}U_3 \nu) + \nabla_{\varepsilon}P=&  f&\text{ in }(0,T)\times\Omega, \\ 
\divergence_{\varepsilon} U =& \ 0&\text{ in }(0,T)\times\Omega ,\\
U(0)=&\ U_0&\text{ in }\Omega ,\\
P \text{ periodic in }x,y,z & \text{ even }&\text{ in }z,
\\
(U_1,U_2) ,U_3 \text{ periodic in }x,y,z,&\text{ even and odd }&\text{ in }z, \text{ respectively.}\\
\end{array}\right .
\end{equation}
To re-obtain \eqref{eq_Diff} later in Proposition~\ref{prop:estimate_F},
we will
insert  in \eqref{eq:transport}
\begin{align}\label{eq:nuomega}
U=(V_{\varepsilon,\delta},W_{\varepsilon,\delta}), \quad\nu=(v,\varepsilon w), \quad\omega= w + \tfrac{1}{\varepsilon}W_{\varepsilon,\delta}, \quad \hbox{and}\quad 
 f=f_{\varepsilon, \delta}
 \end{align}
 with
 \begin{align}\label{eq:fvarepsilon} 
 f_{\varepsilon, \delta}=
\divergence_{H}\left(
(v,\varepsilon w)\otimes V_{\varepsilon,\delta} + (V_{\varepsilon, \delta}, W_{\varepsilon,\delta})\otimes (v+V_{\varepsilon, \delta})
\right)
+(\delta \partial_z^2v, \varepsilon(\partial_t w-\Delta_\delta w + u\cdot \nabla w))^T. 
\end{align}
Here,  using the representation \eqref{eq:Fdivform}
\begin{multline*}
\begin{bmatrix}
F^H_{\varepsilon,\delta}(V_{\varepsilon,\delta}, W_{\varepsilon,\delta}) \\
F_{\varepsilon,\delta}^z(V_{\varepsilon,\delta}, W_{\varepsilon,\delta})
\end{bmatrix}
= -\divergence_H \left((v, \varepsilon w)\otimes  V_{\varepsilon,\delta}\right)
-\divergence_H \left(( V_{\varepsilon,\delta}, W_{\varepsilon,\delta})\otimes (v+ V_{\varepsilon,\delta}) \right)
\\
-\partial_z(\tfrac{1}{\varepsilon}W_{\varepsilon,\delta}(v, \varepsilon w))
-\partial_z((w+\tfrac{1}{\varepsilon}W_{\varepsilon,\delta})( V_{\varepsilon,\delta}, W_{\varepsilon,\delta}))
+
\begin{bmatrix}
-\delta\partial_z^2 v
\\
 \varepsilon(\partial_t w + u \cdot \nabla w - \Delta_\delta w)
\end{bmatrix},
\end{multline*}
where the $\partial_z$-terms are incorporated into the left hand side of \eqref{eq:transport} while the remainder becomes the right hand side \eqref{eq:fvarepsilon}  of \eqref{eq:transport}.

We will interpret equation~\eqref{eq:transport} in the triple of spaces induced by the Gelfand-triple in $xy$-variables
\begin{align*}
V &= \{U\in H^1_zH^1_{xy}\colon (U_1,U_2) \hbox{ periodic and  even and }U_3 \hbox{ periodic and odd in }z\}, 
\\
H &= \{U\in H^1_zL^2_{xy}\colon (U_1,U_2) \hbox{ periodic and  even and }U_3 \hbox{ periodic and  odd in }z\} , \quad 
\hbox{and} \\ 
V'  &= \{U\in H^1_zH^{-1}_{xy}\colon (U_1,U_2) \hbox{ periodic and  even and }U_3 \hbox{ periodic and  odd in }z\}, 
\end{align*}
where we consider here only real-valued functions.
\begin{proposition}\label{prop:transport_lin}
	Let 
	\begin{align*} 
	U_0\in H, \quad f\in L^2(0,T;V'), \quad \partial_z\omega \in L^2(0,T;H), \hbox{ and }  \nu \in L^4(0,T;H^2_zL^4_{xy}).
	\end{align*}
	Then there exists a constant $C>0$ depending only on $\Omega$ and independent of $\varepsilon, \delta$ and the given data, such that if $U\in \IE_1(T)$ is a solution to ~\eqref{eq:transport}, then
	\begin{align*}
	\norm{U}^2_{L^\infty(0,T;H)} +\norm{U}^2_{L^2(0,T;V)} \leq C (\norm{U_0}_H^2 + \norm{f}_{L^2(0,T;V')}^2)e^{C(T+\norm{\partial_z \omega}^2_{L^2(0,T;H)} + \norm{\nu}^4_{L^4(0,T;H^2_zL^4_{xy})})}.
	\end{align*}
\end{proposition}
\begin{proof}
\textit{Step 1 ($L^2$-estimate):}
Multiplying the first equation in \eqref{eq:transport} by $U$ and integrating over $\Omega$ gives
\begin{align}\label{eq:L2energy}
\tfrac{1}{2} \tfrac{d}{dt} \norm{U}_{L^2}^2 + \norm{\nabla_{\delta} U}^2_{L^2} = \langle f,U \rangle_{L^2} 
- \langle \partial_z(\omega U),U \rangle_{L^2}
- \langle \partial_z (\tfrac{1}{\varepsilon} U_3  \nu),U\rangle_{L^2},
\end{align}
where we integrated by parts and by the duality pairing in $L^2_zH^1_{xy}$  
and Young's inequality
\begin{align*}
\langle f,U \rangle_{L^2} 
\leq \norm{f}_{L^2_zH^{-1}_{xy}}\norm{U}_{L^2_zH^{1}_{xy}} 
\leq
\tfrac{1}{12}(\norm{U}_{L^2}^2 + \norm{\nabla_H  U}_{L^2}^2) + 3\norm{f}_{L^2_zH^{-1}_{xy}}^2.
\end{align*}
Integration by parts, anisotropic Hölder's inequalities \eqref{eq:Holder}, the 2-dimensional Ladyzhenskaya's inequality \eqref{eq:Ladyzhenskaya}, and Young's inequality yield
\begin{align*}
\langle \partial_z(\omega U),U \rangle_{L^2} &=
\langle (\partial_z\omega) U,U \rangle_{L^2}
+\tfrac{1}{2}\langle \omega,\partial_z |U|^2 \rangle_{L^2}\\
 &=\tfrac{1}{2} \langle (\partial_z\omega),|U|^2 \rangle_{L^2} \\
&\leq \tfrac{1}{2}  \norm{\partial_z\omega}_{L^\infty_zL^2_{xy}}\norm{U}^2_{L^2_zL^4_{xy}} \\
&\leq C \norm{\partial_z\omega}_{L^\infty_zL^2_{xy}}\norm{U}_{L^2_zL^2_{xy}}\norm{U}_{L^2_zH^1_{xy}} \\
&\leq \tfrac{1}{12}\norm{\nabla_H U}^2_{L^2} + (\tfrac{1}{12}+3C^2 \norm{\partial_z\omega}^2_{L^\infty_zL^2_{xy}})\norm{U}^2_{L^2_zL^2_{xy}}.  
\end{align*}
	For the term $	\partial_z (\tfrac{1}{\varepsilon} U_3 \nu)$, one has by $\divergence_{\varepsilon} U=0$ and the parity conditions for $U$ that
	\begin{align*}
	\partial_z \tfrac{1}{\varepsilon}U_3 = -\divergence_H (U_1,U_2) \quad \hbox{and} \quad \tfrac{1}{\varepsilon}U_3(x,y,z) = -\int_{-1}^z\divergence_H (U_1,U_2)(x,y,\xi) d\xi, \quad  (x,y,z)\in \Omega,
	\end{align*}
which implies
\begin{align*}
\langle \partial_z (\tfrac{1}{\varepsilon}U_3 \nu),U \rangle_{L^2} 
&= \langle (\tfrac{1}{\varepsilon}\partial_z  U_3) \nu,U \rangle_{L^2}
+
\langle \tfrac{1}{\varepsilon}  U_3 \partial_z\nu,U \rangle_{L^2} \\
&= -\langle \nu\divergence_H (U_1,U_2),U \rangle_{L^2}
-
\langle \partial_z\nu \int_{-1}^z\divergence_H (U_1,U_2)(\cdot,\cdot,\xi) d\xi 
 ,U \rangle_{L^2}.
\end{align*}	
One estimates similar to the above using in addition the one-dimensional Sobolev embedding $H^1_z\hookrightarrow L^\infty_z$
\begin{align*}
|\langle \nu\divergence_H (U_1,U_2),U \rangle_{L^2}| &\leq \norm{\nabla_H U}_{L^2_zL^2_{xy}} \norm{\nu}_{L^\infty_zL^4_{xy}} \norm{U}_{L^2_zL^4_{xy}} \\
&\leq C\norm{\nabla_H U}_{L^2} \norm{\nu}_{H^1_zL^4_{xy}} \norm{U}_{L^2}^{1/2}
\norm{U}_{L^2_zH^1_{xy}}^{1/2}, \\
&= C \norm{\nu}_{H^1_zL^4_{xy}} 
(\norm{U}_{L^2}^4\norm{\nabla_H U}_{L^2}^4+\norm{U}_{L^2}^2\norm{\nabla_H U}_{L^2}^6)^{1/4}, \\
|\langle \partial_z\nu\int_{-1}^z\divergence_H (U_1,U_2)(\cdot,\cdot,\xi) d\xi 
,U \rangle_{L^2}
|&\leq \norm{\int_{-1}^z\divergence_H (U_1,U_2)(x,y,\xi) d\xi}_{L^\infty_{z}L^2_{xy}} \norm{\partial_z \nu}_{L^2_z L^4_{xy}}\norm{U}_{L^2_z L^4_{xy}} \\
&\leq C \norm{\nabla_H U}_{L^2_zL^2_{xy}} \norm{\partial_z \nu}_{L^2_z L^4_{xy}} \norm{U}_{L^2}^{1/2}
\norm{U}_{L^2_zH^1_{xy}}^{1/2}\\
&= C \norm{\nu}_{H^1_zL^4_{xy}} 
(\norm{U}_{L^2}^4\norm{\nabla_H U}_{L^2}^4+\norm{U}_{L^2}^2\norm{\nabla_H U}_{L^2}^6)^{1/4},
\end{align*}
and hence by Young's inequality for a constant $C>0$
\begin{align*}
|\langle \partial_z (\varepsilon^{-1}U_3 \nu),U \rangle_{L^2}| \leq
\tfrac{1}{12}\norm{\nabla_H U}_{L^2}^2 + C(1+ \norm{\nu}^4_{H^1_z L^4_{xy}})\norm{U}^2_{L^2}.
\end{align*}
Applying the above estimates to the right hand side in \eqref{eq:L2energy} and absorbing the terms with $\norm{\nabla_HU}_{L^2}^2$ into the left hand side gives for a  constant $C>0$
\begin{align*}
\frac{1}{2} \frac{d}{dt}\norm{U}_{L^2}^2 + \frac{3}{4}\norm{\nabla_\delta U}_{L^2}^2 \leq C \norm{f}_{L^2_zH^{-1}_{xy}}^2
+ C(1+ \norm{\nu}^4_{H^1_z L^4_{xy}}+ \norm{\omega_z}^2_{L^\infty_zL^4_{xy}})\norm{U}^2_{L^2}.
\end{align*}
\textit{Step 2 ($H$-estimate):}
Differentiating equation \eqref{eq:transport} with respect to $z$, multiplying the resulting equation by $\partial_z U$ and integrating over $\Omega$ gives
\begin{align}\label{eq:Henergy}
\frac{1}{2} \frac{d}{dt}\norm{U}_{L^2}^2 + \norm{\nabla_\delta \partial_z U}_{L^2}^2 = \langle \partial_z f, \partial_zU\rangle_{L^2}
-\langle \partial_z^2 (\omega U), \partial_zU\rangle_{L^2}   
-\langle \partial_z^2 (\tfrac{1}{\varepsilon}U_3\nu), \partial_zU\rangle_{L^2}.   
\end{align} 
The first term can be estimated by the duality pairing in $L_z^2H_{xy}^1$ and Young's inequality by \begin{align*}
\langle \partial_z f, \partial_zU\rangle_{L^2} &\leq \norm{\partial_z f}_{L^2_zH^{-1}_{xy}} \norm{\partial_z U}_{L^2_zH^{1}_{xy}} \\
&\leq \tfrac{1}{12} \norm{\nabla_H U}_{H_z^1L^2_{xy}}^2+ \tfrac{1}{12} \norm{U}_{H_z^1L^2_{xy}}^2    + C\norm{\partial_z f}_{L^{2}_zH_{xy}^{-1}}^2
\end{align*}
for some $C>0$.
For the second term note that $\partial_z^2 (\omega U)= (\partial_z^2\omega) U + 2  (\partial_z\omega) (\partial_z U) + \omega (\partial_z^2 U)$, then by the anisotropic Hölder's and Ladyzhenskaya inequalities \eqref{eq:Holder} and \eqref{eq:Ladyzhenskaya}, respectively,  and $H^1_z\hookrightarrow L^\infty_z$ 
\begin{align*}
\langle (\partial_z^2\omega) U, \partial_zU\rangle_{L^2} &\leq \norm{\partial_z^2\omega}_{L_z^2L^2_{xy}} \norm{U}_{L^\infty_zL^4_{xy}}  
\norm{\partial_z U}_{L^2_zL^4_{xy}}\\
&\leq \norm{\partial_z^2\omega}_{L_z^2L^2_{xy}} \norm{U}_{H^1_zL^2_{xy}}\norm{U}_{H^1_zH^1_{xy}}.  
\end{align*}
Then, using integration by parts
\begin{align*}
\langle 2  (\partial_z\omega) (\partial_z U) + \omega (\partial_z^2 U), \partial_zU\rangle_{L^2} &= 
2\langle   (\partial_z\omega) (\partial_z U),\partial_z U\rangle + \langle\omega, \tfrac{1}{2}\partial_z|\partial_z U|^2\rangle_{L^2} \\
&=\tfrac{3}{2} \langle  \partial_z\omega \partial_z U, \partial_zU\rangle_{L^2} \\
&\leq \tfrac{3}{2} \norm{\partial_z \omega}_{L^\infty_z L^2_{xy}}
\norm{\partial_z U}_{L^2_z L^4_{xy}}^2 \\
&\leq C \norm{\partial_z \omega}_{H^1_z L^2_{xy}}
\norm{U}_{H^1_zL^2_{xy}}\norm{U}_{H^1_zH^1_{xy}}.
\end{align*}
Hence,
\begin{align*}
\langle \partial_z^2 (\omega U), \partial_zU\rangle_{L^2} &\leq \tfrac{1}{12}\norm{\nabla_H U}_{H}^2+  C (1+\norm{\partial_z \omega}_{H^1_{z}L^2_{xy}}^2) \norm{U}_{V}\norm{U}_{H}^2.  
\end{align*}
Now, in the  third term using the product rule and the condition $\divergence_\varepsilon U=0$
\begin{align*}
\partial_z^2 (\tfrac{1}{\varepsilon}U_3\nu) &= 
 \tfrac{1}{\varepsilon}U_3\partial_z^2\nu
+ \tfrac{2}{\varepsilon}\partial_zU_3 \partial_z\nu
+ \tfrac{1}{\varepsilon}(\partial_z^2U_3) \nu \\
&= -\int_{-1}^{z}\divergence_H (U_1,U_2)(\cdot,\cdot,\xi)d\xi\partial_z^2\nu
- 2 \divergence_H (U_1,U_2) \partial_z\nu
- \partial_z \divergence_H (U_1,U_2) \nu.
\end{align*}
Then, the third term can be estimated similarly as above using the embedding $H^1_z\hookrightarrow L^\infty_z$ by
\begin{align*}
\langle \tfrac{1}{\varepsilon}U_3\partial_z^2\nu, \partial_zU\rangle_{L^2} 
&\leq C\norm{\nabla_H U}_{L_z^2L^2_{xy}}\norm{\nu}_{H^2_{z}L^4_{xy}} \norm{\partial_z U}_{L^2_{z}L^4_{xy}}, \\
 \langle  \tfrac{2}{\varepsilon}\partial_zU_3 \partial_z\nu, \partial_zU\rangle_{L^2}&\leq 2\norm{\nabla_H U}_{L^\infty_zL_{xy}^2} \norm{\partial_z \nu}_{L^2_zL^4_{xy}}\norm{\partial_z U}_{L^2_zL^4_{xy}} \\
 &\leq 
 C\norm{\nabla_H U}_{H^1_zL^2_{xy}} \norm{ \nu}_{H^1_zL^4_{xy}}\norm{U}_{H^1_zL^2_{xy}}^{1/2}\norm{U}_{H^1_zH^1_{xy}}^{1/2}, \\
 \langle  \tfrac{1}{\varepsilon}(\partial_z^2U_3) \nu, \partial_zU\rangle_{L^2}&\leq
 \norm{\partial_z \divergence_H U}_{L^2_zL^2_{xy}}\norm{\nu}_{L^\infty_{z}L^4_{xy}} \norm{\partial_z U}_{L^2_{z}L^4_{xy}} \\
 &\leq C \norm{\partial_z\nabla_H U}_{L^2} \norm{\nu}_{H^1_zL^4_{xy}} \norm{U}_{H^1_zL^2_{xy}}^{1/2}\norm{U}_{H_z^1H^1_{xy}}^{1/2}. 
\end{align*}
Together with Young's inequality 
\begin{align*}
\langle \partial_z^2 (\tfrac{1}{\varepsilon}U_3\nu), \partial_z U\rangle_{L^2} \leq \tfrac{1}{12}\norm{\nabla_H U}_{H^1_zL^2_{xy}}^2 + C(1+ \norm{\nu}_{H^1_zL^4_{xy}}^4)\norm{U}_{H^1_{z}L^2_{xy}}^2,
\end{align*}
and adding all three terms and compensating the terms with $\norm{\nabla_H U}_{H^1_zL^2_{xy}}^2$ into the left hand side of \eqref{eq:Henergy} gives
\begin{align*}
\frac{1}{2} \frac{d}{dt}\norm{\partial_z U}_{L^2}^2 + \frac{3}{4}\norm{\nabla_\delta \partial_z U}_{L^2}^2 \leq  
C(1+ \norm{\nu}_{H^2_zL^4_{xy}}^4 +\norm{\partial_z\omega}_{H_z^1L^2_{xy}}^2)\norm{U}_{H^1_{z}L^2_{xy}}^2 + C\norm{\partial_z f}_{L^2_zH^{-1}_{xy}}
\end{align*} 
which implies together with the $L^2$-estimate the claim via Gr\"onwall's inequality.
\end{proof}

\subsection{Embeddings for space-time-spaces}

\begin{lemma}\label{lemma:emdebbingIE1L4}
	Let $T>0$, then there exists a constant $C>0$ such that 
	\begin{align*}
	\norm{v}_{L^4(0,T;H^1_zL^4_{xy})} &\leq C\norm{v}_{\IE_1(T)} \quad \hbox{for all }v\in \IE_1(T),\\
	\norm{v}_{L^4(0,T;H^1_zL^4_{xy})} &\leq C\norm{v}_{\IE_z(T)} \quad \hbox{for all } v\in \IE_z(T).
	\end{align*}
\end{lemma}
\begin{proof}	
	The claim for $v\in\IE_1(T)$ follows from the mixed derivative theorem, cf. e.g. \cite[Corollary 4.5.10]{pruss2016moving}, with $\theta=1/4$ in the maximal $L^2$-regularity space and with $\theta=2/3$ to show that $H^{3/2}(\Omega)\hookrightarrow H^{1}_zH^{1/2}_{xy}$,  and Sobolev embeddings which gives the continuous embeddings
	\begin{align*}
	\IE_1(T) \hookrightarrow H^{1/4}(0,T;H^{3/2}(\Omega))\hookrightarrow
	H^{1/4}(0,T;H^{1}_zH^{1/2}_{xy})
	\hookrightarrow
	H^{1/4}(0,T;H^{1}_zL^{4}_{xy}).  
	\end{align*}
	Applying Ladyzhenskaya's inequality \eqref{eq:Ladyzhenskaya} with respect to the $xy$-variables, implies
	\begin{align*}
	\int_0^T \norm{v(t)}_{H^1_zL^4_{xy}}^4 dt &\leq C\int_0^T \norm{v(t)}_{H^1_zH^1_{xy}}^2 \norm{v(t)}_{H^1_zL^2_{xy}}^2 dt \\
	&\leq C\norm{v(t)}_{L^2(0,T;H^1_zH^1_{xy})}^2 \norm{v(t)}_{L^\infty(0,T;H^1_zL^2_{xy})}^2,
	\end{align*}
	and then  taking the forth root and using Young's inequality gives the claim for $v\in\IE_z(T)$.
\end{proof}

\section{Non-linear estimates}\label{sec:non_lin}

\subsection{Bi-linear estimates for the Navier-Stokes and primitive equations}
The non-linearities of the Navier-Stokes, the primitive and  the difference equations involve bilinear terms of the types discussed in the  following lemmata. 
\begin{lemma}[cf. Lemma 4.2 in \cite{convergence}]\label{lemma:NavierStokes_nonlinear}
	Let $T\in (0,\infty)$, then there is a constant $C>0$ such that for all $v_1,v_2\in \IE_1(T)$
	\begin{align*}
	\norm{v_1\partial_i v_2}_{\IE_0} \leq C \norm{v_1}_{\IE_1(T)} \norm{v_2}_{\IE_1(T)},\quad\hbox{where } i\in \{x,y,z\}.   
	\end{align*}
\end{lemma}
\begin{proof}[Sketch of the proof of Proposition~\ref{prop:NS_eps}] 
	Lemma~\ref{lemma:NavierStokes_nonlinear} implies together with Proposition~\ref{prop:maxreg} by the contraction mapping principle the local existence and uniqueness of solutions to \eqref{eq:NS_eps} in maximal $L^2$-$L^2$-regularity spaces. 
\end{proof}

\begin{proof}[Sketch of the proof of Proposition~\ref{prop:NS_H}] 
	Similarly to the above, the local existence and uniqueness of solutions to \eqref{eq:NS_eps3} in maximal $L^2$-$L^2$-regularity spaces follows. Using the energy equality and weak-strong-uniqueness results, this solution can be extended to a global one. 
	The statement for the scaled Stokes equations follows from Proposition~\ref{prop:2DNS}.
\end{proof}

\begin{lemma}\label{lemma:divhv1v2}
	Let $T>0$, then there exists a constant $C>0$ such that 
\begin{align*}
\norm{\divergence_H (v_1 v_2)}_{L^2(0,T;H^1_zH^{-1}_{xy})} \leq  \norm{v_1}_{L^4(0,T;H^1_zL^4_{xy})} \norm{v_2}_{L^4(0,T;H^1_zL^4_{xy})}, 
\end{align*}
for all $v_1 \in L^4(0,T;H^1_zL^4_{xy})$ and  $v_2\in L^4(0,T;H^1_zL^4_{xy})^2$
\end{lemma}
\begin{proof}
	One estimates using that $H^1_z$ is an algebra
	\begin{align*}
	\norm{\divergence_H (v_1 v_2)}_{L^2(0,T;H^1_zH^{-1}_{xy})} \leq \norm{v_1 v_2}_{L^2(0,T;H^1_zL^2_{xy})} \leq  \norm{v_1}_{L^4(0,T;H^1_zL^4_{xy})}\norm{v_2}_{L^4(0,T;H^1_zL^4_{xy})}. \qquad \qedhere
	\end{align*}
\end{proof}
Recall that $\tfrac{1}{\varepsilon}w(v)$ is given by \eqref{eq:Weps} for $v\in \IE_H(T), \IE_z(T), \IE_1(T)$ since with this regularity $\divergence_{H}v$ is integrable with respect to the $z$-variable for almost all $t\in 0,T$ and $(x,y)\in G$. 
\begin{lemma}\label{lemma:vnablaw}
		Let $T>0$ and $\varepsilon>0$, then there exists a constant $C>0$ such that 
		\begin{align*}
	\norm{v_1 \cdot \nabla_H v_2}_{\IE_0(T)} &\leq C \norm{v_1}_{\IE_H(T)} \norm{v_2}_{\IE_z(T)} \quad \hbox{for all }v_1\in \IE_H, v_2\in \IE_z(T), \\
	\norm{\tfrac{1}{\varepsilon}w(v_1) \cdot \partial_z v_2}_{\IE_0(T)} &\leq C \norm{v_1}_{\IE_H(T)} \norm{v_2}_{\IE_z(T)} \quad \hbox{for all } v_1\in \IE_H, v_2\in \IE_z(T), \\
	\norm{v_1 \cdot \nabla_H v_2}_{\IE_0(T)} &\leq C \norm{v_1}_{\IE_H(T)} \norm{v_2}_{\IE_1(T)}\quad \hbox{for all } v_1\in \IE_H, v_2\in \IE_1(T), \\ 
	\norm{\tfrac{1}{\varepsilon}w(v_1) \cdot \partial_z v_2}_{\IE_0(T)} &\leq C \norm{v_1}_{\IE_H(T)} \norm{v_2}_{\IE_1(T)}\quad \hbox{for all } v_1\in \IE_H, v_2\in \IE_1(T).
	\end{align*}
	\end{lemma}

\begin{proof}
By anisotropic Hölder's inequality in time and space, the embedding $H^1_z\hookrightarrow L_z^\infty$,  the Poincar\'e inequality for $w(v_1)$, and \eqref{eq:Weps}
	\begin{align*}
	\norm{v_1 \cdot \nabla_H v_2}_{\IE_0(T)} & \leq \norm{v_1}_{L^4(0,T;L^\infty_zL^4_{xy})}
	\norm{\nabla_Hv_2}_{L^4(0,T;L^2_zL^4_{xy})}, \quad \hbox{and}\\
	\norm{\tfrac{1}{\varepsilon}w(v_1) \cdot \partial_z v_2}_{\IE_0(T)} & \leq \norm{\tfrac{1}{\varepsilon}w(v_1)}_{L^4(0,T;L^\infty_zL^4_{xy})}
	\norm{\partial_z v_2}_{L^4(0,T;L^2_zL^4_{xy})} \\
	& \leq \norm{\nabla_H v_1}_{L^4(0,T;L^2_zL^4_{xy})}
	\norm{\partial_z v_2}_{L^4(0,T;L^2_zL^4_{xy})}.
	\end{align*}
	Now the remaining estimates follow by the mixed derivative theorem and Sobolev embeddings
	\begin{align}\label{eq:embedding2}
	\IE_H(T) \hookrightarrow H^{1/4}(0,T;L^2_zH^{3/2}_{xy}) \hookrightarrow L^{4}(0,T;L^2_zH^{1,4}_{xy}),
	\end{align}
	 and Lemma~\ref{lemma:emdebbingIE1L4}.
\end{proof}

\begin{lemma}\label{lemma:unablaw}
	Let 
	$u=(v,w(v))\in \IE_1(T)$,  then there is a constant $C>0$ such that
	\begin{align*}
	\norm{u\cdot \nabla w(v)}_{L^2(0,T;H^1_zH^{-1}_{xy})} \leq C
		\norm{u}_{L^\infty(0,T;H^1(\Omega))}
			\norm{u}_{L^2(0,T;H^2(\Omega))}.
	\end{align*}
\end{lemma}
\begin{proof}
	Note that using that $\partial_z w=-\divergence_H v$ one obtains
	\begin{align*}
	u \cdot \nabla w= v_1\cdot \partial_x w + v_2\cdot \partial_y w + w\cdot \partial_z w 
	= \divergence_H(wv)- w\divergence_H v + \tfrac{1}{2}\partial_z w^2
	= \divergence_H(wv) +\partial_z w^2.
	\end{align*}
	Next, one estimates each term separately, first
	\begin{align*}
		\norm{\divergence_H(wv)}_{L^2(0,T;H^1_zH^{-1}_{xy})} 
		&\leq C \norm{\partial_z(wv)}_{L^2(0,T;L^2_zL^{2}_{xy})} \\
		&\leq C \norm{v\divergence_H v}_{L^2(0,T;L^2_zL^{2}_{xy})} + 
		C\norm{w\partial_z v}_{L^2(0,T;L^2_zL^{2}_{xy})}\\
		&\leq C \norm{v}_{L^2(0,T;L^\infty(\Omega))}\norm{\nabla_H v}_{L^\infty(0,T;L^2(\Omega))} + 
		C\norm{w}_{L^2(0,T;L^\infty(\Omega)}\norm{v}_{L^\infty(0,T;H^1(\Omega)}\\
		&\leq C
		\norm{u}_{L^\infty(0,T;H^1(\Omega))}
		\norm{v}_{L^2(0,T;H^2(\Omega))},
	\end{align*}
	where one uses the definition of the $H^{-1}_{xy}$-norm, Poincar\'{e}'s inequality for $\partial_z(wv)$,  Hölder's inequality and Sobolev embeddings.
	Second, using again Poincar\'{e}'s inequality now applied to $\partial_z w^2=2w\partial_z w$ and the embedding $H_{xy}^{-1}\hookrightarrow L^{6/5}_{xy}$ one obtains
	\begin{align*}
	\norm{\partial_z w^2}_{L^2(0,T;H^1_zH^{-1}_{xy})} &\leq C 	\norm{\partial_z^2 w^2}_{L^2(0,T;L^2_zL^{6/5}_{xy})}  \\
	&\leq C 	\norm{(\partial_z w)^2}_{L^2(0,T;L^2_zL^{6/5}_{xy})} 
	+ C 	\norm{w\partial_z^2 w}_{L^2(0,T;L^2_zL^{6/5}_{xy})} \\
	&\leq C 	\norm{\partial_z w}_{L^\infty(0,T;L^2_zL^{2}_{xy})} 	\norm{\partial_z w}_{L^2(0,T;L^\infty_zL^{3}_{xy})} 
	+ C 	\norm{w}_{L^\infty(0,T;L^\infty_zL^{3}_{xy})} \norm{\partial_z^2 w}_{L^2(0,T;L^2_zL^{2}_{xy})}  \\
	&\leq C 	\norm{v}_{L^\infty(0,T;H^1(\Omega))} 	\norm{u}_{L^2(0,T;H^2(\Omega))} 
	+ C 	\norm{u}_{L^\infty(0,T;H^1(\Omega))} \norm{v}_{L^2(0,T;H^2( \Omega))},
	\end{align*}
	where one uses Hölder's inequality, the emdeddings $H^1(\Omega)\hookrightarrow  H^{2/3}_zH^{1/3}_{xy}\hookrightarrow L^\infty_zL^{3}_{xy}$, and $\partial_z w=-\divergence_H v$. Combining the above inequalities  the claim follows.
\end{proof}
\subsection{Estimates for right hand side of the difference equations~\eqref{eq_Diff}}
For given
$U_{\varepsilon,\delta}=(V_{\varepsilon,\delta}, W_{\varepsilon,\delta})$ and $u=(v,w)$
consider the term $f_{\varepsilon,\delta}$. This term appeared in \eqref{eq:fvarepsilon} when linearizing the difference equation~\eqref{eq_Diff} as a diffusion-transport type equation in Subsection~\ref{subsec:vertical_reg}.
Recall that
\begin{align*}
f_{\varepsilon,\delta}= \divergence_H \left( (v,\varepsilon w)\otimes V_{\varepsilon,\delta} + U_{\varepsilon,\delta} \otimes (v+ V_{\varepsilon,\delta} ) \right) +
(\delta \partial_z^2 v, \varepsilon (\partial_t w - \Delta_{\delta} w + u  \cdot  \nabla w)). 
\end{align*}
\begin{proposition}\label{prop:estimate_f}
	Let $T>0$, $\varepsilon\in (0,1]$, $\delta>0$, then there exists a constant $C>0$ independent of $U_{\varepsilon,\delta}$, $u$, $\delta$ and $\varepsilon$ such that for $f_{\varepsilon,\delta}$ given in \eqref{eq:fvarepsilon}
	\begin{align*}
	\norm{f_{\varepsilon,\delta}}_{L^2(0,T;V')} \leq C \left(\norm {U_{\varepsilon,\delta}}_{\IE_z(T)}^2  
	+ \norm {u}_{\IE_1(T)}  \norm {U_{\varepsilon,\delta}}_{\IE_z(T)} 
	+ \delta \norm {v}_{L^2(0,T;H^3_zL^2_{xy})} + \varepsilon(\norm{u}_{\IE_1(T)}+\norm{u}^2_{\IE_1(T)})\right).  
	\end{align*}
\end{proposition}	
\begin{proof}
	The  $\divergence_{H}$-terms can be estimated using Lemma~\ref{lemma:divhv1v2} and Lemma~\ref{lemma:emdebbingIE1L4} to obtain
	  \begin{align*}
	  \norm{\divergence_H ( (v,\varepsilon w)\otimes V_{\varepsilon,\delta} + U_{\varepsilon,\delta} \otimes (v+ V_{\varepsilon,\delta} ))}_{L^2(0,T;V')} \leq C( \norm {U_{\varepsilon,\delta}}_{\IE_z(T)}^2  
	  + \norm {u}_{\IE_1(T)}  \norm {U_{\varepsilon,\delta}}_{\IE_z(T)}).
	  \end{align*}
	For the remaining terms recall that it follows from the representation \eqref{eq:Weps} and the Poincar\'e inequality for $w(v)$ that
	\begin{align*}
	\norm{\partial_tw(v)}_{V'}
	&\leq C \norm{\partial_zw(\partial_t v)}_{L^2_zH^{-1}_{xy}}
	\leq C \norm{\partial_tv}_{L^2(\Omega)},\\
	\norm{\partial_z^2w(v)}_{V'}&\leq C \norm{\partial_z^2\divergence_{H} v}_{L^2_zH^{-1}_{xy}}\leq C
	\norm{\partial_z^2 v}_{L^2(\Omega)}, \quad \hbox{and}\\ 
	\norm{\Delta_H w(v)}_{V'}&
	\leq C \norm{\partial_z w(\Delta_Hv)}_{L^2_zH^{-1}_{xy}}
	\leq
	C\norm{\Delta_H v}_{L^2(\Omega)}. 
	\end{align*}
		Thereby, we estimate 
	\begin{align*}
	\norm{\partial_t w- \Delta_\delta w}_{L^2(0,T;V')} \leq C\norm{v}_{\IE_1(T)}.
	\end{align*}
 The non-linear term is estimated by Lemma~\ref{lemma:unablaw}, that is, there is some $C>0$ such that
	\begin{align*}
	\norm{u\cdot \nabla w}_{L^2(0,T;H^1_zH^{-1}_{xy})} \leq 
	C	\norm{u}_{L^\infty(0,T;H^1(\Omega))}
	\norm{u}_{L^2(0,T;H^2(\Omega))}\leq C \norm{u}^2_{\IE_1(T)},
	\end{align*}
	where in the last step we used the trace embedding $\IE_1(T)\hookrightarrow L^\infty(0,T;H^1(\Omega))$. Eventually, since $L^2_{xy}\hookrightarrow  H^{-1}_{xy}$ one has 
$\norm {v}_{L^2(0,T;V')} \leq C \norm {v}_{L^2(0,T;H^3_zL^2_{xy})}$ for some $C>0$.
\end{proof}


	Eventually, combining the previous estimates, we obtain estimates on $F_H$ and $F_z$ appearing as right hand sides \eqref{eq:F} in the difference equation \eqref{eq_Diff}. 
	Recall that
	\begin{align*}
	F_H(V_{\varepsilon,\delta}, W_{\varepsilon,\delta})=&- ( V_{\varepsilon,\delta}, \tfrac{1}{\varepsilon}W_{\varepsilon,\delta}) \cdot \nabla v 
	-u \cdot \nabla V_{\varepsilon,\delta}
	-( V_{\varepsilon,\delta}, \tfrac{1}{\varepsilon}W_{\varepsilon,\delta}) \cdot \nabla V_{\varepsilon,\delta}
	+ \delta\partial_z^2 v,\\
	F_z(V_{\varepsilon,\delta}, W_{\varepsilon,\delta})=&- ( V_{\varepsilon,\delta}, \tfrac{1}{\varepsilon}W_{\varepsilon,\delta}) \cdot \nabla \varepsilon w 
	-u \cdot \nabla W_{\varepsilon,\delta}
	-(V_{\varepsilon,\delta}, \tfrac{1}{\varepsilon} W_{\varepsilon,\delta}) \cdot \nabla W_{\varepsilon,\delta}
	- \varepsilon(\partial_t w + u \cdot \nabla w - \Delta_\delta w).
	\end{align*}
	
\begin{proposition}\label{prop:estimate_F}	Let $T>0$, $\varepsilon\in (0,1]$, $\delta>0$ and
	\begin{align*}
	U_{\varepsilon,\delta}=(V_{\varepsilon,\delta},W_{\varepsilon,\delta})\in\IE_{H,z}(T):= \IE_z(T)\cap\IE_H(T), \quad \hbox{and}\quad u=(v,w(v))\in \IE_1(T). 
	\end{align*}
	 Then there is a $C>0$ independent of $\varepsilon$, $\delta$, $u$, and $U_{\varepsilon,\delta}$ such that
	\begin{multline*}
	\norm{(F_H, F_z)}_{\IE_0(T)} \leq C \big( \norm{(V_{\varepsilon,\delta},  W_{\varepsilon,\delta})}_{\IE_{H,z}(T)}^2 
	+  \norm{(V_{\varepsilon,\delta},  W_{\varepsilon,\delta})}_{\IE_{H,z}(T)}\norm{(v,w(v))}_{\IE_1(T)} \\ 
	+\varepsilon  (\norm{w}_{\IE_1(T)}+\norm{w}_{\IE_1(T)}^2)\big)	+\delta \norm{v}_{\IE_1(T)}.
	\end{multline*}
\end{proposition}
\begin{proof}
By Lemma~\ref{lemma:vnablaw}
\begin{multline*}
	\norm{F_H}_{\IE_0(T)} \leq C\left( \norm{V_{\varepsilon,\delta}}_{\IE_H(T)}\norm{v}_{\IE_1(T)}
	+
\norm{v}_{\IE_1(T)}\norm{V_{\varepsilon,\delta}}_{\IE_z(T)}
+
	 \norm{V_{\varepsilon,\delta}}_{\IE_H(T)} \norm{V_{\varepsilon,\delta}}_{\IE_z(T)}
	 \right)
	 \\
	+	\delta\norm{v}_{\IE_1(T)}.
\end{multline*}
Similarly, by Lemma~\ref{lemma:vnablaw} and Lemma~\ref{lemma:NavierStokes_nonlinear}
\begin{multline*} 
\norm{F_z}_{\IE_0(T)} \leq C\{ \norm{V_{\varepsilon,\delta}}_{\IE_H(T)}\norm{w}_{\IE_1(T)}
+
\norm{u}_{\IE_1(T)}\norm{W_{\varepsilon,\delta}}_{\IE_z(T)}
+\norm{W_{\varepsilon,\delta}}_{\IE_H(T)}\norm{V_{\varepsilon,\delta}}_{\IE_z(T)}
\\
+
\varepsilon(\norm{w}_{\IE_1(T)}+\norm{w}_{\IE_1(T)}+\norm{u}_{\IE_1(T)}^2) 
\}
\end{multline*}
and combining both estimates gives the assertion. 
\end{proof}

\subsection{Estimates for the right hand side of the difference equations~\eqref{eq:F_Diff2}}
Recall that the right hand side  of \eqref{eq:NS_eps4} is given by \eqref{eq:F_Diff2} using  the  notations \eqref{eq:VW_2} and \eqref{eq:U_2}, that is,    
\begin{align*}
\overline{F}_{\varepsilon,\delta}^H=&
-\overline{V}_{\varepsilon,\delta}\cdot\nabla_H \overline{V}_{\varepsilon,\delta}
-\overline{v}_{0,\infty}\cdot\nabla_H \overline{V}_{\varepsilon,\delta}
-\overline{V}_{\varepsilon,\delta}\cdot\nabla_H \overline{v}_{0,\infty}
\\ 
&-\tfrac{1}{2}\int_{-1}^{+1}(\tilde{V}_{\varepsilon,\delta}+\tilde{v}_{0,\infty}^{\varepsilon,\delta})\cdot\nabla_H (\tilde{V}_{\varepsilon,\delta}+\tilde{v}^{\varepsilon,\delta}_{0,\infty}) +
(\tfrac{1}{\varepsilon}W_{\varepsilon,\delta}+w^{\varepsilon,\delta}_{0,\infty})\cdot\partial_z (\tilde{V}_{\varepsilon,\delta}+\tilde{v}^{\varepsilon,\delta}_{0,\infty})
, \\
\tilde{F}_{\varepsilon,\delta}^H=&-(\tilde{V}_{\varepsilon,\delta}+\tilde{v}^{\varepsilon,\delta}_{0,\infty})\cdot\nabla_H (\overline{V}_{\varepsilon,\delta}+\overline{v}_{0,\infty})
-(\overline{V}_{\varepsilon,\delta}+\overline{v}_{0,\infty})\cdot\nabla_H (\tilde{V}_{\varepsilon,\delta}+\tilde{v}^{\varepsilon,\delta}_{0,\infty})
\\
&-(\tilde{V}_{\varepsilon,\delta}+\tilde{v}^{\varepsilon,\delta}_{0,\infty})\cdot\nabla_H (\tilde{V}_{\varepsilon,\delta}+\tilde{v}^{\varepsilon,\delta}_{0,\infty})
-(\tfrac{1}{\varepsilon}W_{\varepsilon,\delta}+w^{\varepsilon,\delta}_{0,\infty})\cdot\partial_z (\tilde{V}_{\varepsilon,\delta}+\tilde{v}^{\varepsilon,\delta}_{0,\infty})
\\
&+ \tfrac{1}{2}\int_{-1}^{+1}(\tilde{V}_{\varepsilon,\delta}+\tilde{v}^{\varepsilon,\delta}_{0,\infty})\cdot\nabla_H (\tilde{V}_{\varepsilon,\delta}+\tilde{v}^{\varepsilon,\delta}_{0,\infty})
+(\tfrac{1}{\varepsilon}W_{\varepsilon,\delta}+w^{\varepsilon,\delta}_{0,\infty})\cdot\partial_z (\tilde{V}_{\varepsilon,\delta}+\tilde{v}^{\varepsilon,\delta}_{0,\infty}),
\\
\tilde{F}_{\varepsilon,\delta}^w=&-(\overline{V}_{\varepsilon,\delta}+\overline{v}_{0,\infty})\cdot \nabla_H(  W_{\varepsilon,\delta}+\varepsilon w^{\varepsilon,\delta}_{0,\infty})
\\
&-(\tilde{V}_{\varepsilon,\delta}+\tilde{v}^{\varepsilon,\delta}_{0,\infty})\cdot \nabla_H (W_{\varepsilon,\delta}+\varepsilon w^{\varepsilon,\delta}_{0,\infty}))
-(\tfrac{1}{\varepsilon}W_{\varepsilon,\delta}+ w^{\varepsilon,\delta}_{0,\infty})\cdot \partial_z (W_{\varepsilon,\delta}+\varepsilon w_{0,\infty}^{\varepsilon,\delta}).
\end{align*}

\begin{lemma}\label{lemma:vnablaw2}
	Let $T>0$, $\varepsilon\in (0,1]$, $q=2$ and $p=2$, then there exists a constant $C>0$ such that for all $v_1,v_2\in L^4(0,T;H^{3/2}(\Omega)\cap L^2_{\sigma,\varepsilon}(\Omega))$
	\begin{align*}
		\norm{v_1 \cdot \nabla_H v_2}_{\IE_0(T)} &\leq C \norm{v_1}_{L^4(0,T;H^{3/2})} \norm{v_2}_{L^4(0,T;H^{3/2})}, \\
		\norm{\tfrac{1}{\varepsilon}w(v_1) \cdot \partial_z v_2}_{\IE_0(T)} &\leq C \norm{v_1}_{L^4(0,T;H^{3/2})} \norm{v_2}_{_{L^4(0,T:H^{3/2})}},
	\end{align*}
	and there exists a constant $C>0$ such that 
	\begin{align*}
	\norm{v}_{L^4(0,T;H^{3/2})} \leq C \norm{v}_{\IE_1(T)} \quad \hbox{for all } v\in \IE_1(T).
	\end{align*}	
\end{lemma}

\begin{proof}
	As in the proof of Lemma~\ref{lemma:vnablaw}
	\begin{align*}
		\norm{v_1 \cdot \nabla_H v_2}_{\IE_0(T)} & \leq \norm{v_1}_{L^4(0,T;L^\infty_zL^4_{xy})}
		\norm{v_2}_{L^4(0,T;L^2_zL^4_{xy})}, \quad \hbox{and}\\
		\norm{\tfrac{1}{\varepsilon}w(v_1) \cdot \partial_z v_2}_{\IE_0(T)} 
		& \leq \norm{\nabla_H v_1}_{L^4(0,T;L^2_zL^4_{xy})}
		\norm{\partial_z v_2}_{L^4(0,T;L^2_zL^4_{xy})}.
	\end{align*}
	Then one uses the Sobolev embeddings and the mixed derivative theorem to obtain  
	\begin{align*}
		H^{3/2}(\Omega) \hookrightarrow
		L^2_zH^{3/2}_{xy} \hookrightarrow
		L^2_zH^{1,4}_{xy}   \quad \hbox{and}\quad
		H^{3/2}(\Omega)\hookrightarrow H^1_zH^{1/2}_{xy} \hookrightarrow 
			H^1_zL^{4}_{xy}  
 	\end{align*}
	Also, by the mixed derivative theorem and Sobolev embeddings
	\begin{align*}
		\IE_1(T) \hookrightarrow H^{1/4}(0,T;H^{3/2}(\Omega)) \hookrightarrow L^{4}(0,T;H^{3/2}(\Omega)). \qquad \qedhere
	\end{align*}
\end{proof}

\begin{proposition}\label{prop:F_diff2}
	Let $T>0$ and $\varepsilon\in (0,1]$, $\delta>0$, then there exists a constant $C>0$ independent of $\varepsilon,\delta$ such that
	\begin{align*}
		\norm{\overline{F}_{\varepsilon,\delta}^H}_{\IE_0(T)}	& \leq C \big(\norm{U_{\varepsilon,\delta}}_{\IE_1(T)}^2
		+\norm{U_{\varepsilon,\delta}}_{\IE_1(T)}
		\norm{u^{\varepsilon,\delta}_{0,\infty}}_{L^4(0,T;H^{3/2})}
		+\norm{\tilde{u}^{\varepsilon,\delta}_{0,\infty}}_{L^4(0,T;H^{3/2})}^2\big), \\
		 	\norm{	\tilde{F}_{\varepsilon,\delta}^H}_{\IE_0(T)}	& \leq
 C \big(\norm{U_{\varepsilon,\delta}}_{\IE_1(T)}^2
+\norm{U_{\varepsilon,\delta}}_{\IE_1(T)}
\norm{u^{\varepsilon,\delta}_{0,\infty}}_{L^4(0,T;H^{3/2})}
+\norm{u^{\varepsilon,\delta}_{0,\infty}}_{L^4(0,T;H^{3/2})}
\norm{\tilde{u}^{\varepsilon,\delta}_{0,\infty}}_{L^4(0,T;H^{3/2})}
\big), \\
		 		\norm{\tilde{F}_{\varepsilon,\delta}^w}_{\IE_0(T)}	& \leq
		 		  C \big(\norm{U_{\varepsilon,\delta}}_{\IE_1(T)}^2
		 		+\norm{U_{\varepsilon,\delta}}_{\IE_1(T)}
		 		\norm{u^{\varepsilon,\delta}_{0,\infty}}_{L^4(0,T;H^{3/2})}
		 		+\norm{\tilde{u}^{\varepsilon,\delta}_{0,\infty}}_{L^4(0,T;H^{3/2})}^2
		 		\big).
	\end{align*}
\end{proposition}
\begin{proof}
The terms without prefactor $\tfrac{1}{\varepsilon}$ and $\varepsilon$ can be estimated directly by Lemma~\ref{lemma:NavierStokes_nonlinear} and Lemma~\ref{lemma:vnablaw2}. The terms with prefactor $\tfrac{1}{\varepsilon}$ 
are of the form $\tfrac{1}{\varepsilon}W_{\varepsilon,\delta} \partial_z \tilde{V}_{\varepsilon,\delta}$ and $\tfrac{1}{\varepsilon}W_{\varepsilon,\delta} \partial_z W_{\varepsilon,\delta}
$.
These can be estimated by
 Lemmas~\ref{lemma:vnablaw} and~\ref{lemma:vnablaw2}. This concludes the proof.	
\end{proof}

\section{Additional regularity}\label{sec:additional_reg}
In this section we discuss the proof of Proposition~\ref{prop:PE_H}. As a first step we state the following additional regularity results on the solution of \eqref{eq_PEH} dealing with the preservation of regularity in the vertical directions. Their proofs rely on energy estimates and are not given here but are elaborated in detail in \cite{Wrona_PhD}.  
\begin{proposition}[Proposition 4.3.26 in \cite{Wrona_PhD}]\label{prop:HNL2}
	Let	
	\begin{align*}
		N\in \IN, \quad \hbox{and}\quad v_0\in H_z^NL^2_{xy}\cap H^1_{\overline{\sigma},\eta}(\Omega) \quad \hbox{for}\quad \eta>0. 	
	\end{align*}
	Then the strong solution $v$ to \eqref{eq_PEH} has the additional regularity
	\begin{align*}
		v\in L^\infty(0,T;H^N_zL^2_{xy})\quad \hbox{and}\quad 
		\nabla_H v\in L^2(0,T;H^N_zL^2_{xy}).
	\end{align*}
\end{proposition}
\begin{proof}[Sketch of the proof]
	Applying $\partial_z^r$ to the velocity equation~\eqref{eq_PEH} for $r\leq N$ and testing with $\partial_z^r v$ gives
	\begin{align*}
	\tfrac{d}{dt} \norm{\partial_z^r v}_{L^2}^2 + 2\norm{\partial_z^r\nabla_H v}_{L^2}^2  = - 2\langle \partial_z^r(v\cdot\nabla_H v + w\partial_z v) , \partial_z^rv \rangle_{L^2}.
	\end{align*}
This can be estimated as in \cite[Proposition 4.3.7.]{Wrona_PhD} for some constant $C_r>0$ by
	\begin{align*}
	\tfrac{d}{dt} \norm{\partial_z^r v}_{L^2}^2 + 2\norm{\partial_z^r\nabla_H v}_{L^2}^2  \leq C_r 
	(1+ \norm{v}_{H_z^{r-1}H_{xy}^1}^2\norm{v}_{H_z^{r-1}L_{xy}^2}^2)
	- \langle \partial_z^r(v\cdot\nabla_H v + w\partial_z v)\norm{\partial_z^r v}_{L^2}^2 , \partial_z^rv \rangle_{L^2}.
	\end{align*}
	Then the claim follows by Gr\"onwall's inequality. 
\end{proof}

\begin{proposition}[Proposition 4.3.27 in \cite{Wrona_PhD}]\label{prop:L2L4}
	Let	
	\begin{align*}
		v_0\in H^1_{\overline{\sigma},\eta}(\Omega) \quad \hbox{for}\quad \eta>0 \quad \hbox{with}\quad \partial_z v_0\in H^1(\Omega). 	
	\end{align*}
	Then the strong solution $v$ to \eqref{eq_PEH} has the additional regularity
	\begin{align*}
		v\in L^\infty(0,T;H^1_zH^1_{xy})\quad \hbox{and}\quad 
		\nabla_H v\in L^2(0,T;H^1_zH^1_{xy}).
	\end{align*}
	In particular $v\in L^4(0,T;H^1_zH^{1,4}_{xy})$.
\end{proposition}
\begin{proof}[Sketch of the proof]
		Applying $\partial_z$ to the velocity equation~\eqref{eq_PEH}  and testing with $\partial_z \Delta_H v$ gives
	\begin{align*}
	\tfrac{d}{dt} \norm{\partial_z \nabla_H v}_{L^2}^2 + 2\norm{\Delta_H v}_{L^2}^2  = - 2\langle \partial_z(v\cdot\nabla_H v + w\partial_z v) , \partial_z \Delta_H v \rangle_{L^2}.
	\end{align*}
		This can be estimated  by
	\begin{multline*}
\tfrac{d}{dt} \norm{\partial_z \nabla_H v}_{L^2}^2 + \norm{\partial_z\Delta_H v}_{L^2}^2  \leq C 
	\norm{v}_{H^1}(1+ \norm{v}_{H^1}+	\norm{\Delta_Hv}_{L^2})\norm{\nabla_Hv}_{H}^2 \\
	+\norm{\nabla_H\partial_z^2v}_{L^2}^2
	+\norm{\partial_z^2v}_{L^2}^2\norm{\nabla_Hv}_{L^2}^2(\norm{\nabla_Hv}_{L^2} + \norm{\Delta_Hv}_{L^2})^2.
	\end{multline*}
	By the estimates on $\norm{\nabla_Hv}_{L^2}$ and $\norm{\Delta_H^2v}_{L^2}$ in 
	\cite[Proposition 5.8]{HusseinSaalWrona} as well as Proposition~\ref{prop:HNL2} with $N=2$ the relevant terms on the left hand side are integrable. Hence, the claim follows by Gr\"onwall's inequality.  
\end{proof}

To derive additional regularity on $w(v)$,
let $v$ be a solutions to the primitive equations with horizontal viscosity~\eqref{eq_PEH} and $p$ the associated pressure. Then one can apply  $\int_{-1}^z-\divergence_H\cdot$ to \eqref{eq_PEH} and consider $w=w(v)$ which solves adding on both side of \eqref{eq_PEH} the term $-\partial_z^2 w=\partial_z \divergence_H v$ 
%
\begin{align*}
\partial_t w - \Delta w = \int_{-1}^z \divergence_H (\nabla_H p + (v, w)\cdot \nabla v ) + \partial_z \divergence_H v, \quad w(0)=w_0.
\end{align*}
This equation  allows us to derive maximal regularity estimates for $w$ provided $v$ is sufficiently regular. 
Consider for given $v, p$ and $\omega_0$
\begin{align}\label{eq:omega}
	\partial_t \omega - \Delta \omega = \int_{-1}^z \divergence_H (\nabla_H p + (v, \omega)\cdot \nabla v ) + \partial_z \divergence_H v, \quad \omega(0)=\omega_0.
\end{align}
\begin{lemma}\label{lemma:regw}
	Let  $v\in \IE_1^v(T)$, $p\in \IE_0^p(T)$, 
	and $\omega_0\in H^1(\Omega)$,
	then   there is a unique solution $\omega\in \IE_1(T)$ to the equation \eqref{eq:omega}.
\end{lemma}

\begin{proof}
	The equation \eqref{eq:omega} can be rewritten as
	\begin{align*}
		\partial_t \omega - \Delta \omega - B(v,\omega) = f(v)+\partial_z \divergence_H v, \quad \omega(0)=w_0, \quad \hbox{where } B(v,\omega) = - \int_{-1}^z \divergence_H[(v, \omega)\cdot \nabla v]
	\end{align*}
	and applying $\divergence_H\overline{\cdot}$ to \eqref{eq_PEH} it follows that 
	\begin{align*}
		-\Delta_H p = \tfrac{1}{2} \divergence_H\int_{-1}^1 [(v,w(v))\cdot \nabla v] \quad \hbox{and hence} \quad 
		f(v)=\tfrac{z+1}{2} \int_{-1}^1 \divergence_H [(v,w(v))\cdot \nabla v].
	\end{align*}
		Since $v\in \IE_1(T)$, one estimates 
	\begin{align*}
		\norm{\partial_z \divergence_H v}_{\IE_0} \leq \norm{v}_{\IE_1}.
	\end{align*}
	The remaining terms can be estimated exactly as done in the proof of \cite[Proposition 4.5]{convergence} which gives the claim. 
\end{proof}


\begin{proof}[Proof of Proposition~\ref{prop:PE_H}]
	The existence and uniqueness statement in part (a) is taken from \cite{CaoLiTiti_PHE_2017}.

Let $v_0$ be as in Proposition~\ref{prop:PE_H} (b) and $v\in \IE_H(T)\cap\IE_z(T)$ be the solution to \eqref{eq_PEH} with initial condition $v_0$ from Proposition~\ref{prop:PE_H}  (a). 
Then by Proposition~\ref{prop:HNL2}
\begin{align}\label{eq:regv}
	v\in \IE_H(T)\cap\IE_z(T)\cap L^2(0,T;H^2_zH^1_{xy})\cap L^\infty(0,T;H_z^2L^2_{xy}) \hookrightarrow \IE_1(T)\cap L^\infty(0,T;H_z^2L^2_{xy}).
\end{align}
Thereby, Lemma~\ref{lemma:regw} is applicable, and there is a unique solution $\omega\in \IE_1(T)$ to \eqref{eq:omega}. Now, we have to show that indeed $\omega =w(v)$.  Integrating over the first equation in \eqref{eq_PEH} with respect to the $z$-variable and testing the resulting equation with $\nabla_H(w(v)-\omega)$ we find
\begin{multline*}
\langle \int_{-1}^z \partial_t v, \nabla_H(w(v)-\omega) \rangle_{L^2} + \langle \nabla w, \nabla(w(v)-\omega) \rangle_{L^2} \\
=\langle \int_{-1}^z (v,w(v))\cdot \nabla v +\nabla_H p -\partial_z^2v, \nabla_H(w(v)-\omega) \rangle_{L^2}.
\end{multline*}
Similarly, testing now \eqref{eq:omega} with $w(v)-\omega$
gives
\begin{multline*}
	\langle \partial_t \omega, w(v)-\omega \rangle_{L^2} + \langle \nabla \omega, \nabla(w(v)-\omega) \rangle_{L^2}
	=\langle \divergence_{H}\int_{-1}^z (v,\omega)\cdot \nabla v +\nabla_H p -\partial_z^2v, w(v)-\omega \rangle_{L^2}.
\end{multline*}
Integrating by parts and taking the difference gives
\begin{multline*}
\tfrac{1}{2}	\partial_t \norm{w-\omega}^2_{L^2} +\norm{\nabla (w-\omega)}^2_{L^2}=
	\langle (w(v)-\omega)\partial_z v, \nabla_H(w(v)-\omega) \rangle_{L^2} \\
	\leq \norm{(w(v)-\omega)}_{L^2}^{1/2}\norm{(w(v)-\omega)}_{L^2_zH^1_{xy}}^{1/2}
	\norm{\partial_z v}^{1/2}_{L^2_zH^1_{xy}}	\norm{\partial_z v}^{1/2}_{H^1_zL^2_{xy}}
	\norm{\nabla_H(w(v)-\omega)}_{L^2} \\
\leq\norm{(w(v)-\omega)}_{L^2}^{2}
  C(1+
		\norm{\partial_z v}^{2}_{L^2_zH^1_{xy}}	\norm{\partial_z v}^{2}_{H^1_zL^2_{xy}})
		+	\tfrac{1}{2}\norm{\nabla_H(w(v)-\omega)}_{L^2}^2,
\end{multline*}
where we used \cite[Lemma 6.2 b)]{HusseinSaalWrona} and Young's inequality. By \eqref{eq:regv}
\begin{align*}
	\norm{\partial_z v}^{2}_{L^2_zH^1_{xy}}	\norm{\partial_z v}^{2}_{H^1_zL^2_{xy}} \in L^\infty(0,T)\subset L^1(0,T),
\end{align*}
and hence by Gr\"onwall's inequality $w(v)=\omega\in \IE_1(T)$
by Lemma~\ref{lemma:regw}.	
	
	Part (c) of Proposition~\ref{prop:PE_H}  follows directly from Proposition~\ref{prop:L2L4}, and part (d) from Proposition~\ref{prop:HNL2}.
	\end{proof} 



\begin{remark}[Regularity assumptions in Propositions~\ref{prop:transport_lin} and \ref{prop:estimate_F}]
To apply Proposition~\ref{prop:transport_lin} to the situation of \eqref{eq_Diff}, we choose as indicated in \eqref{eq:nuomega} and \eqref{eq:fvarepsilon} the functions 
\begin{align*}
\nu=(v,\varepsilon w), \quad \omega= w + \tfrac{1}{\varepsilon}W_{\varepsilon,\delta}, \quad\hbox{and}\quad f=f_{\varepsilon,\delta}.
\end{align*}
Moreover, we need some regularity results on $u=(v,w(v))$ to estimate Proposition~\ref{prop:estimate_F}.
\begin{enumerate}[(a)]
	\item To apply Proposition~\ref{prop:estimate_f}, we need $u\in \IE_1(T)$ which holds by Proposition~\ref{prop:PE_H}~(b).
	To control the term $\delta \norm {v}_{L^2(0,T;H^3_zL^2_{xy})}$ in Proposition~\ref{prop:estimate_f}, we need Proposition~\ref{prop:PE_H} (d).
	\item The regularity assumption on $\nu$ in Proposition~\ref{prop:transport_lin} follows  from Proposition~\ref{prop:PE_H} (c) using \eqref{eq:Weps}.
	For the choice \eqref{eq:nuomega} by \eqref{eq:Weps} $\partial \omega= -\divergence_{H} v - \divergence_{H} V_{\varepsilon,\delta}$. 
	\item By Proposition~\ref{prop:PE_H} (a) $v\in \IE_1(T)$, assuming $V_{\varepsilon,\delta}\in \IE_1(T)$, and by the mixed derivative theorem  $\IE_1(T)\hookrightarrow L^2(0,T;H^1_zH^1_{xy})$, then the regularity of $\partial_z \omega$ in Proposition~\ref{prop:transport_lin}  
	is assured.
\item Proposition~\ref{prop:PE_H} (b) allows us also to control $\norm{u}_{\IE_1(T)}$ on the right hand side of the estimate in Proposition~\ref{prop:estimate_F}.	
\end{enumerate}

\end{remark}

\section{Quadratic inequalities and proof of the convergences}
\label{sec:quadratic_inequality}

\subsection{Boundedness by quadratic inequalities and maximal existence intervals}

\begin{lemma}\label{lemma:qudineq_1}
	For $a,b \in  \IR$ with $a<b$ let $X\in C([a,b);[0,\infty))$. If there exists a constant $C>0$ such that for $0<\varepsilon <\tfrac{1}{16C}$ the function $X$ satisfies
	 \begin{align*}
	     X(t)  \leq C X^2(t) + \tfrac{1}{2} X(t) + \varepsilon \quad \hbox {for all } t\in [a,b), \quad \hbox{and} \quad
	     X(a) \leq \tfrac{1}{4C},
	 \end{align*}
	 then
	 \begin{align*}
	 \sup_{t\in [a,b)} X(t) \leq 4 \varepsilon.
	 \end{align*}
\end{lemma}
\begin{proof}
	From the assumption it follows that
	\begin{align*}
	0 \leq  X^2(t) - \tfrac{1}{2C} X(t) + \tfrac{\varepsilon}{C}, \quad  t\in [a,b),
	\end{align*}
	and since $\tfrac{1}{(4C)^2}  - \tfrac{\varepsilon}{C}>0$ due to the assumption $\varepsilon <\tfrac{1}{16C}$ the solution to this quadratic inequality is for $t\in [a,b)$  either 
	\begin{align*}
	X(t) &\leq  \tfrac{1}{4C} - \left( \tfrac{1}{(4C)^2}  - \tfrac{\varepsilon}{C} \right)^{1/2}=\tfrac{\varepsilon}{C}\left( \tfrac{1}{4C} + \left( \tfrac{1}{(4C)^2}  - \tfrac{\varepsilon}{C} \right)^{1/2}\right)^{-1}\leq 4 \varepsilon \quad \hbox{or} \\
		X(t) &\geq  \tfrac{1}{4C} + \left( \tfrac{1}{(4C)^2}  - \tfrac{\varepsilon}{C} \right)^{1/2}.
	\end{align*}
	Due to the continuity of $X$ the set $X([a,b))\subset [0,\infty)$ is connected, and hence only one of the two possibilities can occur, and since $X(a)\leq  \tfrac{1}{4C}$, the first inequality holds for all $t\in [a,b)$.
\end{proof}

\begin{lemma}\label{lemma:qudineq_2}
		For $a,b \in  \IR$ with $a<b$ let $X\in C([a,b);[0,\infty))$. If there exist constants $C,K>0$ such that for $0<\varepsilon <\min \{\tfrac{1}{64C}, \tfrac{\ln(3/2)}{8K}\}$ the function $X$ satisfies
	\begin{align*}
	X(t)  \leq \left(C X^2(t) + \tfrac{1}{4} X(t) + \varepsilon \right) e^{KX(t)} \quad \hbox {for all } t\in [a,b), \quad \hbox{and} \quad
	X(a) \leq \min\{\tfrac{1}{8C}, \tfrac{\ln(3/2)}{K}\},
	\end{align*}
	then
	\begin{align*}
	\sup_{t\in [a,b)} X(t) \leq 8 \varepsilon.
	\end{align*}
\end{lemma}
\begin{proof}
	Consider
	\begin{align*}
	t^\ast:=\sup\{t\in [a,b)\colon X(s)\leq  \tfrac{\ln(2)}{K} \hbox{ for all } s\in [a,t]  \},
	\end{align*}
	where, because of $X(a)\leq\tfrac{\ln(3/2)}{K}<\tfrac{\ln(2)}{K}$ and the continuity of $X$, one has $t^\ast>a$.
	
	Assume now that $t^\ast <b$, then by continuity of $X$ and the assumptions on $X$ it follows that $X(t^\ast)=\tfrac{\ln(2)}{K}$ and 
		\begin{align}\label{eq:qineq}
	X(t)  \leq 2C X^2(t) + \tfrac{1}{2} X(t) + 2\varepsilon  \quad \hbox {for all } t\in [a,t^\ast).
	\end{align}
	Applying now Lemma~\ref{lemma:qudineq_1} on $[a,t^\ast)$, one concludes that
	\begin{align*}
	X(t) \leq \sup_{t\in [a,t^\ast)} X(t) \leq 8 \varepsilon \leq \tfrac{\ln(
		3/2)}{K} <\tfrac{\ln(2)}{K}
	\end{align*}
	which is a contradiction to $X(t^\ast)=\tfrac{\ln(2)}{K}$. Hence $t^\ast=b$, and \eqref{eq:qineq} holds on $[a,b)$ which by Lemma~\ref{lemma:qudineq_1} implies that $\sup_{t\in [a,b)} X(t) \leq 8 \varepsilon$.
\end{proof}

\begin{proposition}\label{prop:QI}
	Let $T>0$ be a finite time, and 
	\begin{align*}
	(X_{\eta})_{\eta\in (0,1)}, \quad X_{\eta} \colon [0,T]\rightarrow [0,\infty] \quad \hbox{with} \quad X_{\eta}(0)=0
	\end{align*}
	be a family of increasing functions.  Assume that
	\begin{enumerate}[(a)]
		\item  $(X_{\eta})_{\eta\in (0,1)}$ has the following local existence property
		\begin{equation}\label{eq:LE}
		\left \{\begin{array}{rl}
		&\hbox{For each } \eta\in (0,1) \hbox{ there exists } s_{\eta}^\ast\in (0,T] \hbox{ such that } \\
		&X_{\eta}\in C([0,s_{\eta}^\ast); [0,\infty)),
		\end{array}\right .\tag{LE}
		\end{equation}
		\item $(X_{\eta})_{\eta\in (0,1)}$ has	the following maximal existence property
		\begin{equation}\label{eq:ME}
		\left \{\begin{array}{rl}
		&\hbox{If }  t_{\eta}^\ast:= \sup \{t\in [0,T]\colon X_{\eta}\in C([0,t]; [0,\infty)) \}< T, \\
		&\hbox {then } \sup_{t\in [0,t_{\eta}^\ast)}X_{\eta}=\infty,
		\end{array}\right .\tag{ME}
		\end{equation}
		\item  there are increasing functions 
		\begin{align*}
		G_i\in C([0,T]; [0,\infty)) \quad\hbox{with} \quad g_i(s,t):=G_i(s)-G_i(t) \hbox{ for } s,t\in [0,T], \quad i=1,2,3,
		\end{align*}
		 a decreasing function $f\in  C((0,T]; [0,\infty))$ and constants $k,K>0$ such that $(X_{\eta})_{\eta\in (0,1)}$ satisfies the following quadratic inequality 
		 	\begin{equation}\label{eq:QI}
		 \left \{\begin{array}{rl}
		 &\hbox{If } X_{\eta}\in C([0,t_2); [0,\infty)) \hbox{ for } t_2\in (0,T],   
		 \hbox { then for } t_1\in [0,t_2) \hbox{ and }  \\
		 &X_{\eta}^{t_1}(t):= X_{\eta}(t)-X_{\eta}(t_1)  \hbox{one has for } t\in (t_1,t_2)\\
		 &X_{\eta}^{t_1}(t) \leq \left( k X_{\eta}^{t_1}(t)^2 + g_2(t,t_1)X_{\eta}^{t_1}(t) + \eta g_3(t,t_1)+ f(t-t_1)X_{\eta}(t_1)\right) e^{K X_{\eta}^{t_1}(t) + g_1(t,t_1)} .
		 \end{array}\right .\tag{QI}
		 \end{equation}
	\end{enumerate}
	If these assumptions hold, then there exists $\eta^\ast=\eta^\ast(T,g_1,g_2,g_3, k)\in (0,1)$ such that
	\begin{align*}
	X_{\eta}(t)\in C([0,T]; [0,\infty)) \quad \hbox{and} \quad \max_{t\in [0,T]} X_{\eta}(t) \leq C^\ast\eta \quad\hbox {for all} \quad \eta \in (0,\eta^\ast),
	\end{align*}
where the constant $C^\ast>0$ is independent of $\eta$.
\end{proposition}
\begin{proof}
	First note that since the functions $G_i$, $i=1,2,3$, are uniformly continuous, there exists a $T^\ast$ such that
	\begin{align}\label{eq:g1g2g3}
	g_1(t+T^\ast,t)\leq \ln(2), \quad 	g_2(t+T^\ast,t)\leq \tfrac{1}{8}, \hbox{ and }	g_3(t+T^\ast,t)\leq \tfrac{1}{2} \quad  \hbox{for all } t \in [0,T-T^\ast].
	\end{align}  
	Now, set 
	\begin{align*}
	N:=\lceil \tfrac{T}{T^\ast/2} \rceil \quad \hbox{and} \quad T_n:=
	\begin{cases}
	n(T^\ast/2) & \hbox{if } n < N, \\
	T& \hbox{if } n = N,
	\end{cases} \quad \hbox{for } n=1, \ldots, N.
	\end{align*}
	Then the following inductive statement will be proven for $n=1,\ldots, N$
	\begin{equation}
\left \{\begin{array}{rl}
&X_{\eta}\in  C([0,T_n]; [0,\infty)) \quad \hbox{and there are constants } C_n>0   \hbox{ and }   \eta_n\in (0,1) \hbox{ such that } \\  
&\sup_{t\in [0,T_n]}X_{\eta}(t) \leq C_n \eta \quad \hbox {for } \eta \in (0,\eta_n).
\end{array}\right .
\tag{$A_n$}
\end{equation}

To prove the induction basis $(A_1)$ recall that $X_\eta(0)=0$ and by \eqref{eq:LE} there exists $s^\ast_\eta>0$ such that $X_\eta\in C([0,s_{\eta}^\ast); [0,\infty))$, in particular  $$t_{\eta}^\ast= \sup \{t\in [0,T]\colon X_{\eta}\in C([0,t]; [0,\infty)) \}>0.$$
Now,  applying \eqref{eq:QI} with $t_1=0$ and $t_2=\min\{t^\ast_\eta, T_1\}$	it follows using ~\eqref{eq:g1g2g3} that
\begin{align*}
X_{\eta}(t) \leq \left(2k X_{\eta}(t)^2 + \tfrac{1}{4}X_{\eta}(t) + \eta   \right)e^{K X_{\eta}(t) } \quad \hbox{for } t\in (0,\min\{t^\ast_\eta, T_1\}).
\end{align*}
Therefore applying Lemma~\ref{lemma:qudineq_2} with $0=X_{\eta}(0)\leq \min\{\tfrac{1}{16k}, \tfrac{\ln(3/2)}{K}\}$ implies that 
\begin{align}\label{eq:A1}
\sup_{t\in [0,\min\{t^\ast_\eta, T_1\})} X_{\eta}(t) \leq C_1 \eta \quad \hbox{for } \eta \in (0,\eta_1), \quad \hbox{where } C_1=8 \hbox{ and } \eta_1= \min\{\tfrac{1}{128 k}, \tfrac{\ln(3/2)}{8K}\}.
\end{align}
Assume now that $t^\ast_\eta\leq T_1<T$, then 
\begin{align*}
\sup_{t\in [0,t^\ast_\eta)} X_{\eta}(t)\leq C_1 \eta_1<\infty \quad \hbox{for } \eta \in (0,\eta_1)
\end{align*}
which is a contradiction to assumption \eqref{eq:ME}. 
If already $T_1=T$ and $t^\ast_\eta<T_1=T$, then the same argument given before leads to a contradiction, and if  $t^\ast_\eta=T_1=T$ then $\sup_{t\in [0,T)} X_{\eta}(t) <\infty$, and since $X_\eta$ is increasing, it has a continuous extension to $[0,T]$.
  Hence $X_\eta\in C([0,T_1]; [0,\infty))$ for all $\eta \in (0,\eta_1)$. Then, by continuity the supremum  in \eqref{eq:A1} can be taken in fact over $t\in [0,T_1]$ and hence $(A_1)$ holds.

For the induction step, assume that $(A_n)$ holds for $n<N$. First, one notices that by $(A_n)$ one has
\begin{align*}
\sup_{t\in [0,T_n]}X_\eta(t)\leq C_n \eta_n<\infty \quad \hbox{for } \eta\in (0,\eta_n),
\end{align*} 
and hence by \eqref{eq:ME} it follows that  $t_\eta^\ast>T_n$ for $\eta\in (0,\eta_n)$.


Next, apply the quadratic inequality \eqref{eq:QI} with $t_2=\min\{t_\eta^\ast, T_{n+1}\}$ and $t_1=T_{n-1}$ to obtain
\begin{align*}
X_{\eta}^{t_1}(t) \leq \left( 2k X_{\eta}^{t_1}(t)^2 + \tfrac{1}{4}X_{\eta}^{t_1}(t) + \eta (1+ C_nf(t-t_1) \right) e^{K X_{\eta}^{t_1}(t)} \quad \hbox{for } t\in (t_1,t_2).
\end{align*}
Observe that by the definition of $X_{\eta}^{t_1}$ and the induction hypothesis $(A_n)$
\begin{align*}
X_{\eta}^{t_1}(T_n) \leq X_{\eta}(T_n) \leq 
\min \{\tfrac{1}{16 k}, \tfrac{\ln(3/2)}{K} \} \quad \hbox{for } \eta < \eta_n^1:=\min \{\eta_n, \tfrac{1}{16 k C_n}, \tfrac{\ln(3/2)}{K C_n} \}.
\end{align*}
Hence, applying Lemma~\ref{lemma:qudineq_2} on $[T_n, t_2)$, one finds that
\begin{align}\label{eq:QIX}
X_{\eta}^{t_1}(t) \leq C_{n+1}\eta \quad \hbox{for } \eta \in (0,\eta_{n+1}) \hbox{ and } t\in [T_n,t_2)
\end{align}
with $\eta_{n+1}:=\min\{\eta_{n+1}^1, \eta_{n+1}^2\}$,
\begin{align*}
\eta_{n+1}^2 := \tfrac{1}{1+ C_nf(T^\ast/2)}\min \{ \tfrac{1}{128 k}, \tfrac{\ln(3/2)}{K} \}, \quad \hbox{and}\quad C_{n+1}= 8 (1+ C_nf(T^\ast/2)),
\end{align*}
where one uses that $f$ is decreasing hence and $f(t-T_{n-1})\leq f(T^\ast/2)$ for all $t\in [T_n,t_2]$.

Assume now if $n+1<N$ that $t_\eta^\ast\leq T_{n+1}$ for $\eta \in (0,\eta_{n+1})$, then by \eqref{eq:QIX} one obtains a contradiction to \eqref{eq:ME}, and hence $t_\eta^\ast> T_{n+1}$  and $X_\eta\in C([0,T_{n+1}]; [0,\infty))$ for $\eta \in (0,\eta_{n+1})$. If $n+1=N$, then $t_\eta^\ast< T_{n+1}=T$ for $\eta \in (0,\eta_{n+1})$ leads again to a contradiction, and if  $t_\eta^\ast= T_{n+1}=T$ for $\eta \in (0,\eta_{n+1})$ one concludes using the bound from \eqref{eq:QIX} and the fact that $X_\eta$ is increasing that $X_\eta\in C([0,T_{n+1}]; [0,\infty))$. Hence $(A_{n+1})$ follows.

The claim now follows from $(A_N)$ with $\eta^\ast=\eta_N$ and $C^\ast=C_N$.
\end{proof}


\subsection{Convergence for $\varepsilon\to 0$ and $\delta\to 0$}

\begin{proof}[Proof of Theorem~\ref{thm:main1}]
	Let $(V_{\varepsilon,\delta}, W_{\varepsilon,\delta})$ be the solution of the difference equation~\ref{eq_Diff} solved by \eqref{eq:VW}.
	Then, we set
	\begin{align}\label{eq:Xt}
	X_{\varepsilon,\delta}(t):=\norm{(V_{\varepsilon,\delta}, W_{\varepsilon,\delta})}_{\IE_{H,\delta}(t)}^2
	+\norm{(V_{\varepsilon,\delta}, W_{\varepsilon,\delta})}_{\IE_{z}(t)}^2, \quad \varepsilon,\delta>0,
	\end{align}
	and verify that this satisfies the assumptions of Proposition~\ref{prop:QI}. 
	
	\noindent
	\textit{Step 1:} By the local well-posedness of the scaled Navier-Stokes equations, cf. Proposition~\ref{prop:NS_eps}, and the global well-posedness of the primitive equations with horizontal viscosity, cf. Proposition~\ref{prop:PE_H} (a)--(b), one has property \eqref{eq:LE}, that is, for each $\varepsilon,\delta\in (0,1)$ there is a $T^\ast_{\varepsilon,\delta}$ such that
	\begin{align*}
	X_{\varepsilon,\delta}(T^\ast_{\varepsilon,\delta})\in C([0,T_{\varepsilon,\delta}^\ast);[0,\infty)).
	\end{align*}
	
		\noindent
	\textit{Step 2:} By blow-up criteria for the semi-linear equation~\ref{eq_Diff}, one has that also \eqref{eq:ME} holds. 
	
		\noindent
	\textit{Step 3:} The property \eqref{eq:QI} can be derived from the linear and non-linear estimates derived in Sections~\ref{sec:lin} and \ref{sec:non_lin}, respectively.
	Let $t_2\in (0,T]$,  $t_1\in [0,t_2)$, and, $t\in [0,t_2-t_1)$,  then consider the shifted quantities
	\begin{align*}
	(V^{t_1}_{\varepsilon,\delta}, W^{t_1}_{\varepsilon,\delta})(t)&:=(V_{\varepsilon,\delta}, W_{\varepsilon,\delta})(t+t_1),  \\
	(v^{t_1}, w^{t_1}(t)&:=(v,w)(t+t_1).
	\end{align*}
	Note that by the definition of $X_{\varepsilon,\delta}(t)$ in ~\eqref{eq:Xt}, one has for the quantity $X^{t_1}_{\varepsilon,\delta}(t)$ from \eqref{eq:QI} that for $t+t_1\in (t_1,t_2)$
	\begin{align*}
	X^{t_1}_{\varepsilon,\delta}(t)
	&=X_{\varepsilon,\delta}(t+t_1)-X_{\varepsilon,\delta}(t_1) \\
	&=\norm{(V_{\varepsilon,\delta}, W_{\varepsilon,\delta})}_{\IE_{H,\delta}(t)}^2
	+\norm{(V_{\varepsilon,\delta}, W_{\varepsilon,\delta})}_{\IE_{z}(t)}^2
	-\norm{(V_{\varepsilon,\delta}, W_{\varepsilon,\delta})}_{\IE_{H,\delta}(t_1)}^2
	-\norm{(V_{\varepsilon,\delta}, W_{\varepsilon,\delta})}_{\IE_{z}(t_1)}^2\\
	&=\norm{(V^{t_1}_{\varepsilon,\delta}, W^{t_1}_{\varepsilon,\delta})}_{\IE_{H,\delta}(t)}^2
	+\norm{(V^{t_1}_{\varepsilon,\delta}, W^{t_1}_{\varepsilon,\delta})}_{\IE_{z}(t)}^2, 
	\end{align*}
	where  the norms in the preultimate line are taken on $(0,t)$ and $(0,t_1)$, respectively, and in the last 
	equation norms for the shifted functions $(V^{t_1}_{\varepsilon,\delta}, W^{t_1}_{\varepsilon,\delta})$ are on the interval $(0,t)$ which corresponds  to norms of $(V_{\varepsilon,\delta}, W_{\varepsilon,\delta})$ on $(t_1,t+t_1)$.
	 	
		\noindent
	\textit{Step 3a ($\IE_z$-estimate):} By Proposition~\ref{prop:transport_lin} with
	\begin{align}\label{eq:omeganu}
	\omega=w^{t_1}+ \tfrac{1}{\varepsilon}W^{t_1}_{\varepsilon,\delta}\quad \hbox{and} \quad \nu=(v^{t_1},\varepsilon w^{t_1}),
	\end{align}
	  one obtains
	\begin{multline*}
\norm{(V^{t_1}_{\varepsilon,\delta}, W^{t_1}_{\varepsilon,\delta})}_{\IE_{z}(t)}^2 \leq 	C\Big(\norm{f_{\varepsilon,\delta}(V^{t_1}_{\varepsilon,\delta}, W^{t_1}_{\varepsilon,\delta})}_{L^2(0,t;V^\prime)}^2\\
+\norm{(V^{t_1}_{\varepsilon,\delta}(0), W^{t_1}_{\varepsilon,\delta}(0))}_{H}^2
\Big)e^{C(t+ \norm{V^{t_1}_{\varepsilon,\delta}}^2_{L^2(0,t;V)}+\norm{(v^{t_1},\partial_z v^{t_1}}^2_{\IE_1(t)})}.
	\end{multline*}
	Here, we have applied the higher regularity on $v$ and $w$ from Proposition~\ref{prop:PE_H}.
	More precisely, by Proposition~\ref{prop:PE_H} (a) one has 
	\begin{align*}
			\partial_z w^{t_1}=-\divergence_{H}v^{t_1}\in L^2(0,t;H_{z}^1L^2_{xy}) \quad \hbox{since}\quad  v\in \IE_z\subset L^2(0,t;H_{z}^1H^1_{xy}).
	\end{align*}
	By Step 1,  $\tfrac{1}{\varepsilon}\partial_z W^{t_1}_{\varepsilon,\delta}=-\divergence_{H} V^{t_1}_{\varepsilon,\delta}\in L^2(0,t;H_{z}^1L^2_{xy})$, so that $\omega$ as in \eqref{eq:omeganu} fulfills the assumptions of 
	Proposition~\ref{prop:transport_lin}, that is, 
	\begin{align*}
\partial_z \omega\in L^2(0,t;H) \quad \hbox{and}	\quad	\norm{\partial_z \omega}_{L^2(0,t;H)}\leq C\left( \norm{v}_{\IE_z(t)}
+\norm{V^{t_1}_{\varepsilon,\delta}}_{L^2(0,t;V)}\right).
	\end{align*}
	 By Proposition~\ref{prop:PE_H} (c) and \eqref{eq:Weps}, one finds writing 
	 $$w(\cdot,\cdot,z)=-\int_{-1}^z\divergence_{H} v(\cdot,\cdot,\xi) d\xi\quad\hbox{for}\quad z\in (1,1), \quad \hbox{that}\quad
	 w\in L^4(0,t;H^2_zL^4_{xy}),$$
	and by interpolation it follows from Proposition~\ref{prop:PE_H} (d) that also $v\in L^4(0,t;H^2_zL^4_{xy})$ and hence for $\nu$ as in \eqref{eq:omeganu}
		\begin{align*}
		\nu\in L^4(0,t;H_z^2L_{xy}^4) \quad \hbox{and}	\quad	\norm{\nu}_{L^4(0,t;H_z^2L_{xy}^4)}\leq 
		C \left(\norm{v}_{L^4(0,t;H^2_zL^4_{xy})} +  \norm{v}_{L^4(0,t;H^1_zH^{1,4}_{xy})}  \right).
	\end{align*}
	Moreover, 
	by Proposition~\ref{prop:estimate_f} one can estimate $f_{\varepsilon,\delta}$
	to obtain
%
	\begin{multline*}
\norm{(V^{t_1}_{\varepsilon,\delta}, W^{t_1}_{\varepsilon,\delta})}_{\IE_{z}(t)}^2 \leq 	C\Big(\norm{V^{t_1}_{\varepsilon,\delta}, W^{t_1}_{\varepsilon,\delta})}_{\IE_z(t))}^4
+\norm{(v^{t_1},w^{t_1})}^2_{\IE_1(t)}\norm{V^{t_1}_{\varepsilon,\delta}, W^{t_1}_{\varepsilon,\delta})}_{\IE_z(t))}^2
\\
+ \delta^2 \norm {v^{t_1}}^2_{L^2(0,T;H^3_zL^2_{xy})} + \varepsilon^2 
\norm{u^{t_1}}_{\IE_1(T)}^4 \\
+\norm{(V^{t_1}_{\varepsilon,\delta}(0), W^{t_1}_{\varepsilon,\delta}(0))}_{H}^2
\Big)e^{C(t+ \norm{V^{t_1}_{\varepsilon,\delta}}^2_{L^2(0,t;V)}+\norm{(v^{t_1},\partial_z v^{t_1}}^2_{\IE_1(t)})}.
\end{multline*}
Therefore,
	\begin{multline*}
\norm{(V^{t_1}_{\varepsilon,\delta},\varepsilon W^{t_1}_{\varepsilon,\delta})}_{\IE_{z}(t)}^2 \leq 	C\Big(X^{t_1}_{\varepsilon,\delta}(t)^2
+\norm{(v^{t_1},w^{t_1})}^4_{\IE_1(t)}X^{t_1}_{\varepsilon,\delta}(t)^2
+ \delta^2 \norm {v^{t_1}}^2_{L^2(0,T;H^3_zL^2_{xy})} 
\\+ \varepsilon^2 
\norm{u^{t_1}}_{\IE_1(T)}^4
+X^{t_1}_{\varepsilon,\delta}(t)
\Big)e^{t+ X^{t_1}_{\varepsilon,\delta}(t)+\norm{(v^{t_1},\partial_z v^{t_1}}^2_{\IE_1(t)}}.
\end{multline*}
		\noindent
\textit{Step 3b ($\IE_{H,\delta}$-estimate):} By Proposition~\ref{prop:estimate_F} and Proposition~\ref{prop:maxreg} one estimates
	\begin{multline*}
\norm{(V^{t_1}_{\varepsilon,\delta}, W^{t_1}_{\varepsilon,\delta})}_{\IE_{H}(t)}^2 
\leq 	C\Big(\norm{F_{H}((V^{t_1}_{\varepsilon,\delta}, W^{t_1}_{\varepsilon,\delta})),F_{z}(V^{t_1}_{\varepsilon,\delta}, W^{t_1}_{\varepsilon,\delta})}_{\IE_0(t)}^2+\norm{(V^{t_1}_{\varepsilon,\delta}(0), W^{t_1}_{\varepsilon,\delta}(0))}_{H^1}^2
\Big)\\
\leq
C \norm{(V_{\varepsilon,\delta},  W_{\varepsilon,\delta})}_{\IE_{H,z}(T)}^4 
+ C \norm{(V_{\varepsilon,\delta}, W_{\varepsilon,\delta})}_{\IE_{H,z}(T)}^2\norm{(v,w(v))}_{\IE_1(T)}^2 \\
+\delta^2 \norm{v}_{\IE_1(T)}^2 
+\varepsilon C (\norm{w}_{\IE_1(T)}+\norm{w}_{\IE_1(T)}^2)^2.
\end{multline*}
\noindent
\textit{Step 4:}
Step 3a and 3b together give
\begin{align*}
X^{t_1}_{\varepsilon,\delta}(t)\leq \big(kX^{t_1}_{\varepsilon,\delta}(t)^2+g_2(t,t_1)X^{t_1}_{\varepsilon,\delta}(t)+(\varepsilon+\delta)g_3(t,t_1)+f(t-t_1)X^{t_1}_{\varepsilon,\delta}(t)\big)e^{C X^{t_1}_{\varepsilon,\delta}(t)+ g_1(t,t_1)},
\end{align*}
where
\begin{align*}
g_i(t,t_1)=G_i(t)-G_i(t_1), \quad i\in \{1,2,3\}
\end{align*}
with
\begin{align*}
G_1(t)&:=C(t+\norm{(v,\partial_zv)}^2_{\IE_1(t)}), \\
G_2(t)&:=C\norm{(v,w)}^2_{\IE_1(t)}, \\
G_3(t)&:=C(1+\norm{v}^2_{\IE_z(t)\cap L^\infty(0,t;H^1)}+\norm{w}^2_{\IE_1(t)})\norm{(v,w)}^2_{\IE_1(t)} + C\norm{(v,\partial_zv)}^2_{\IE_1(t)}.
\end{align*}
Hence by Proposition~\ref{prop:QI} one has
\begin{align*}
\max_{t\in [0,T]}X^{t_1}_{\varepsilon,\delta}(t)\leq C(\varepsilon+\delta).
\end{align*}
\end{proof}

\subsection{Convergence for $\varepsilon\to 0$ and $\delta\to \infty$}
\begin{proof}[Proof of Theorem~\ref{thm:main2}]
Consider \eqref{eq:NS_eps4} with 
\eqref{eq:VW_2}, then set
	\begin{align*}
X_{\varepsilon, \delta}(t):=\norm{(V_{\varepsilon,\delta}, W_{\varepsilon,\delta})}_{\IE_{1,\delta}(t)}^2, \quad \varepsilon,\delta>0,
\end{align*}
Estimating the right hand side of  
\eqref{eq:NS_eps4} by 
Proposition~\ref{prop:F_diff2} and Proposition~\ref{prop:2DNS} one gets
for $\varepsilon\in (0,1]$ and $\delta \geq 1$
\begin{align*}
X_{\varepsilon,\delta}(t) \leq 
C\big(X_{\varepsilon,\delta}(t)^2+X_{\varepsilon,\delta}(t)
\norm{u_{0,\infty}^{\varepsilon,\delta}}_{L^4(0,t;H^{3/2}(\Omega))}
 +\norm{u_{0,\infty}^{\varepsilon,\delta}}_{L^4(0,t;H^{3/2}(\Omega))}\tfrac{1}{\delta^{1/4}}
 \big).
\end{align*}
Setting
\begin{align*}
k=C, \quad K>0, \quad f=0, \quad G_1=0, \quad \hbox{and}\quad G_2(t)=G_3(t)=C\sup_{\varepsilon\in (0,1]}\norm{u_{0,\infty}^{\varepsilon,\delta}}_{L^4(0,t;H^{3/2}(\Omega))},
\end{align*}
one can apply Proposition~\ref{prop:QI}
for all $\varepsilon\in (0,1]$, and hence using again Proposition~\ref{prop:2DNS} for $\delta$ sufficiently large
\begin{align*}
\sup_{\varepsilon\in (0,1] }\left(\norm{\overline{v}-\overline{v}_{\varepsilon,\delta}}_{\IE_1} +\norm{\tilde{u}_{\varepsilon,\delta}}_{L^4(0,T;H^{3/2}(\Omega))}\right) \leq \left( 
\sup_{\varepsilon\in (0,1] } X_{\varepsilon,\delta}(T) 
+\norm{\tilde{u}_{0,\infty}^{\varepsilon,\delta}}_{L^4(0,T;H^{3/2}(\Omega))}\right)
\leq\frac{C}{\delta^{1/4}}.
\end{align*} 
\end{proof}

\subsection*{Acknowledgment} We would like to thank anonymous referees for their helpful
	comments, references and suggestions.

\subsection*{Declarations: Funding and/or Conflicts of interests/Competing interests}
None.

\bibliographystyle{abbrv}
\bibliography{literature}
\end{document}